\documentclass[11pt,a4paper]{article}
\usepackage[american]{babel}
\usepackage{verbatim}
\usepackage{a4}
\usepackage{hyperref}
\usepackage{amsmath, amssymb, amsthm}
\usepackage{subfig}
\usepackage{cases}
\usepackage{ifthen}
\usepackage{tikz}
\usetikzlibrary{positioning,arrows,patterns}
\usepackage{tkz-euclide}
\usepackage{dsfont}

\DeclareMathOperator{\re}{Re}
\DeclareMathOperator{\im}{Im}

\newcommand{\mR}{{\mathds R}}

\newtheorem{theorem}{Theorem}[section]
\newtheorem{proposition}[theorem]{Proposition}
\newtheorem{lemma}[theorem]{Lemma}
\newtheorem{corollary}[theorem]{Corollary}

\theoremstyle{definition}

\newtheorem*{definition}{Definition}
\newtheorem*{example}{Example}
\newtheorem*{remark}{Remark}

\begin{document}

\title{The convergence of discrete period matrices}

\author{Felix G\"unther\footnote{Institut f\"ur Mathematik, MA 8-3, Technische
Universit\"at Berlin, Stra{\ss}e des 17. Juni 136, 10623 Berlin, Germany. E-mail: fguenth@math.tu-berlin.de.}}

\date{}
\maketitle

\begin{abstract}
\noindent
We study compact polyhedral surfaces as Riemann surfaces and their discrete counterparts obtained through quadrilateral cellular decompositions and a linear discretization of the Cauchy-Riemann equation. By ensuring uniformly bounded interior and intersection angles of diagonals, we establish the convergence of discrete Dirichlet energies of discrete harmonic differentials with equal black and white periods to the Dirichlet energy of the corresponding continuous harmonic differential with the same periods. This convergence also extends to the discrete period matrix, with a description of the blocks of the complete discrete period matrix in the limit. Moreover, when the quadrilaterals have orthogonal diagonals, we observe convergence of discrete Abelian integrals of the first kind. Adapting the quadrangulations around conical singularities allows us to improve the convergence rate to a linear function of the maximum edge length.
\\ \vspace{0.5ex}

\vspace{0.5ex}

\end{abstract}

\raggedbottom
\setlength{\parindent}{0pt}
\setlength{\parskip}{1ex}


\section{Introduction}\label{sec:intro}

Over the past few decades, discrete complex analysis has made remarkable progress, even though its development started much later compared to its continuous counterparts. While the continuous theory of complex analysis has a long-established history, its discrete analog emerged relatively recently and is still the subject of ongoing research and exploration.

The concept of discrete harmonicity can be traced back to Kirchhoff's circuit laws, but it was not until almost one hundred years ago that authors such as Coruant, Friedrichs, and Lewy \cite{CoFrLe28} delved into the study of discrete harmonic functions on the square lattice. Subsequently, Isaacs \cite{Is41} introduced discrete holomorphic functions on the square lattice, with later investigations by Lelong-Ferrand \cite{Fe44,Fe55} leading to the proof of the Riemann mapping theorem using a limit of discrete holomorphic functions. Expanding the theory, Duffin extended it to rhombic lattices \cite{Du56, Du68}, while Mercat \cite{Me01}, Kenyon \cite{Ke00, Ke02}, and Chelkak and Smirnov \cite{ChSm11,ChSm12} revived the study of discrete holomorphicity on rhombic lattices from the perspective of statistical physics models such as the Ising model.

Different notions of discrete holomorphicity were explored, including those on triangular lattices by Dynnikov and Novikov \cite{DN03} and via circle packings and patterns by Stephenson \cite{Ste05}, Rodin and Sullivan \cite{RSul87}, and B"ucking \cite{Bue08}. Bobenko, Mercat, and Suris demonstrated the connection between discrete holomorphic functions and infinitesimal deformations of circle patterns \cite{BoMeSu05}, while Bobenko, Pinkall, and Springborn developed a theory of discrete conformally equivalent metrics \cite{BoPSp10}. Further, Wilson presented a linear theory for discrete complex analysis on triangulated surfaces using holomorphic cochains \cite{Wi08}, showing convergence of discrete period matrices to their continuous counterparts.

Mercat's work on discrete Riemann surfaces and discrete period matrices based on linear discretization \cite{Me01, Me01b} laid the groundwork for generalizing the theory to discrete Riemann surfaces composed of general quadrilaterals. Later, Skopenkov's proof of the uniform convergence of discrete harmonic functions on planar domains decomposed into quadrilaterals with orthogonal diagonals \cite{Sk13}, inspired Bobenko and Skopenkov to extend the ideas to discrete Riemann surfaces \cite{BoSk12}. They explored discrete harmonic functions on triangulations of polyhedral surfaces and demonstrated convergence of discrete Dirichlet energies, period matrices, and discrete Abelian integrals of the first kind. However, the problem of generalizing the convergence of discrete period matrices to general quadrangulations remained unresolved.

In collaboration with Bobenko, we further developed the linear theory on general bipartite quadrilateral cellular decompositions of Riemann surfaces \cite{BoG15, BoG17}. Our work, based on the medial graph of the decomposition, established an analogy to the continuous theory, enabling discretizations of linear theorems. In this paper, we continue our exploration by solving the open problem stated by Bobenko and Skopenkov, focusing on the convergence of discrete period matrices on compact polyhedral surfaces. To achieve this, we adapt the approach of Bobenko and Skopenkov and draw inspiration from nonconforming finite element methods developed by Braess \cite{Br07}.

In addition to showing convergence of the discrete period matrix, we specify the limits of the four $g \times g$-blocks of the complete discrete period matrix, which also considers the discrete holomorphic differentials that do not have equal black and white periods, providing more information on the underlying discrete Riemann surface. Furthermore, for discrete Riemann surfaces decomposed into orthodiagonal quadrilaterals, we demonstrate how the convergence of discrete Abelian integrals of the first kind follows in a manner similar to \cite{BoSk12}. Moreover, we present improved error estimates in the convergence proofs, achieving a linear rate in the maximum edge length $h$ by utilizing $h$-adapted quadrangulations, as defined in \cite{BoBu17}.

Period matrices play a fundamental role in the study of nonlinear integrable equations \cite{BBEIM94}, making their numerical computation of great interest. Recently, applied algebraic geometry has seen a surge in interest regarding the computation of period matrices. While Riemann surfaces and algebraic curves are equivalent, determining the algebraic curve associated with a Riemann surface is a non-trivial task. This computation involves Riemann theta functions, leading \c{C}elik, Fairchild, and Mandelshtam to coin the term \textit{transcendental divide} for this problem \cite{CFM23}. The period matrix not only determines the Jacobian variety but, according to Torelli's theorem, also the algebraic curve. Although Deconinck and van Hoeij demonstrated practical methods for computing period matrices of algebraic curves \cite{DH01}, reconstructing the algebraic curve from the period matrix has remained limited to surfaces of low genus, thanks to an algorithm developed by Agostini, \c{C}elik, and Eken \cite{ACE21}. However, the theory of discrete Riemann surfaces offers a valuable tool to approximate the period matrix with known error estimates for compact polyhedral surfaces, which are equivalent to compact Riemann surfaces by a theorem of Troyanov \cite{T86}.

In Chapter~\ref{sec:basic}, we recapitulate the basic notions of discrete Riemann surfaces as developed in \cite{BoG17}, specifically focusing on compact polyhedral surfaces. We then delve into the proof of convergence of the discrete Dirichlet energy in Chapter~\ref{sec:convergence1}. Building upon this, we demonstrate how the convergence of the discrete Dirichlet energy implies the convergence of the discrete period matrix to the period matrix of the Riemann surface in Chapter~\ref{sec:convergence2}. Additionally, we address the convergence of discrete Abelian integrals of the first kind in the context of quadrangulations with orthogonal diagonals. Finally, we discuss the application of our convergence results in the paper \cite{CFM23} and outline future research questions in Chapter~\ref{sec:conclusion}.


\section{Discrete Riemann surfaces and their period matrices}\label{sec:basic}

In this chapter, we present a concise summary of the linear theory of discrete Riemann surfaces using general quad-graphs. For a more comprehensive understanding and complete proofs, we recommend referring to the detailed explanations in \cite{BoG17}. Our main emphasis is on the discretization of polyhedral surfaces, which can be viewed as both continuous and discrete Riemann surfaces. However, it is important to note that not all discrete Riemann surfaces can be realized as polyhedral surfaces.

In Section~\ref{sec:Riemann}, we begin by clarifying the type of Riemann surfaces we are discretizing. We define discrete Riemann surfaces based on quad-graphs, which serve as our discretization framework. In Section~\ref{sec:differential}, we delve into the concept of discrete holomorphicity and define discrete differentials on the medial graph of the quad-graph. The definitions and key properties of discrete period matrices are covered in Section~\ref{sec:period}, where we review their fundamental aspects. Finally, in Section~\ref{sec:Dirichlet}, we express both the continuous and discrete Dirichlet energy as quadratic forms. These forms involve entries derived from the sums, products, and inverses of the real and imaginary parts of the continuous or the (complete) discrete period matrix.


\subsection{Discrete Riemann surfaces}\label{sec:Riemann}

Following the work of Bobenko and Skopenkov \cite{BoSk12}, we consider the discretization of a \textit{compact polyhedral surface} $\Sigma$ without boundary. To guarantee the existence of a period matrix, we assume that $\Sigma$ is not topologically equivalent to a sphere. The surface $\Sigma$ is constructed by joining Euclidean polygons along shared edges \cite{BoG17}. In other words, $\Sigma$ possesses a piecewise flat metric, and it exhibits isolated conical singularities \cite{BoSk12}. These conical singularities correspond to vertices where the sum of the interior angles of the adjacent polygons deviates from $2\pi$.

Let $x$ and $y$ be two points on the surface $\Sigma$. The notation $|xy|$ represents the geodesic distance between these two points. We define the finite set $S$ to be the collection of conical singularities of $\Sigma$. Each singularity $O$ in $S$ is associated with an open disk $D_O$ centered at $O$ and has a radius $R_O$. Importantly, the closure of $D_O$ does not contain any other singularity. The singularity index $\gamma_O$ is defined as the ratio of $2\pi R_O$ to the circumference of the boundary of $D_O$. In simpler terms, $2\pi/\gamma_O$ represents the sum of angles at $O$ formed by the polygons incident to this vertex. Geometrically, $D_O$ can be visualized as a cone with an opening angle of $2\pi/\gamma_O$ \cite{BoSk12}.

The surface $\Sigma$ possesses a canonical complex structure \cite{Bo11}. To define an atlas, we introduce \textit{polar coordinates} $(r,\psi)$ around each singularity $O$ on the corresponding disk $D_O$. Here, $r$ ranges from 0 to $R_O$, and $\psi$ lies in the interval $\mathds{R}$ modulo $2\pi/\gamma_O$. We define the chart $g_O: D_O \to \mathds{C}$ as $g_O(r,\psi) = r^{\gamma_O} \exp(i\gamma_O\psi)$, which corresponds to the complex map $z \mapsto z^{\gamma_O}$.

For points $p$ in $\Sigma$ excluding the singularities $S$, we select open disks $D_p \subset \Sigma$ centered at $p$. These disks do not contain any singularities in their closures. We choose an isometry $g_p: D_p \to \mathds{C}$ using the Euclidean metric on $\mathds{C}$. If $D_p$ lies completely within a Euclidean polygon of $\Sigma$, the choice of $g_p$ is straightforward. If $D_p$ intersects an edge, we flatten $D_p$ along that edge. Different isometries $g_p$ and $g_{p'}$ for $p$ and $p'$ in $\Sigma \backslash S$ are related by an isometry of $\mathds{C}$. The transition map $g_O \circ g_p^{-1}$ takes the form $z \mapsto (az+b)^{\gamma_O}$, where $a$ and $b$ are complex numbers with $\left|a\right|=1$. The singularity $O$ is not included in this transition map, specifically excluding $z=-b/a$. Importantly, all transition maps are holomorphic, thus the charts $g_x$ for $x$ in $\Sigma$ define a complex analytic atlas.

Conversely, it was shown by Troyanov \cite{T86} that any compact Riemann surface can be realized as a polyhedral surface.

\begin{definition}
A discretization of the Riemann surface $\Sigma$, or a \textit{discrete Riemann surface} denoted as $(\Sigma,\Lambda)$, is a finite decomposition $\Lambda$ of $\Sigma$ into flat embedded quadrilaterals $F(\Lambda)$. The decomposition satisfies the following properties: The set of conical singularities $S$ is included in the set of vertices $V(\Lambda)$, and the bipartite property holds for the graph $(V(\Lambda),E(\Lambda))$. This decomposition is referred to as a quad-graph, where $V(\Lambda)$, $E(\Lambda)$, and $F(\Lambda)$ represent the sets of vertices, edges, and faces, respectively. The diagonals of the quadrilaterals in $F(\Lambda)$ give rise to two connected graphs called the \textit{black graph} $\Gamma$ and the \textit{white graph} $\Gamma^*$, while $\Diamond$ represents the dual graph of $\Lambda$. If the diagonals of every quadrilateral in $F(\Lambda)$ intersect each other orthogonally, we call the discretization $(\Sigma,\Lambda)$ \textit{orthodiagonal}.
\end{definition}

A \textit{discrete chart} $z = z_Q$ of a face $Q$ refers to an isometric embedding of $Q$ into the complex plane $\mathds{C}$.

\begin{definition}
Given a face $Q \in F(\Lambda)$, its vertices are denoted as $b_-, w_-, b_+$, and $w_+$ in counterclockwise order, with the first vertex being a black vertex. These vertices are identified with their corresponding complex values provided by a discrete chart.

The \textit{discrete complex structure} of $(\Sigma, \Lambda)$ is defined by the set of all $\rho_Q$, where $Q$ ranges over $F(\Lambda)$. The value of $\rho_Q$ is given by \[\rho_Q = -i\frac{w_+ - w_-}{b_+ - b_-}.\]
\end{definition}

In the work presented in \cite{BoG17}, it has been shown that a bijective map $z = z_v$ can always be found from the star of a vertex $v$ in $\Lambda$ to a planar quadrilateral vertex star, preserving the discrete complex structure. This map is also referred to as a \textit{discrete chart}.

It is worth noting that $\rho_Q$ always has a positive real part. In the case where the diagonals $b_- b_+$ and $w_- w_+$ are orthogonal to each other, $\rho_Q \in \mathds{R}^+$. Thus, an orthodiagonal discrete Riemann surface corresponds to a real discrete complex structure.

In their work \cite{BoSk12}, Bobenko and Skopenkov considered a triangulation $\Gamma$ with the property that the set of conical singularities $S$ is contained in the set of vertices $V(\Gamma)$. They placed the dual vertices of $\Gamma$ at the circumcenters of the triangles. When the circumcircles of the triangles do not contain any points other than the three vertices of the triangle, a quad-graph $\Lambda$ is obtained by connecting the vertices of each triangle in $F(\Gamma)$ with the corresponding dual vertex in $V(\Gamma^*)$. The type of triangulation described above is known as a \textit{Delaunay triangulation}, and the resulting orthodiagonal quad-graph $\Lambda$ is referred to as the \textit{Delaunay-Voronoi quadrangulation}.

In our research, we expand upon the convergence results presented in \cite{BoSk12} by considering general quad-graphs, specifically in relation to period matrices. Additionally, we extend the analysis to encompass general orthodiagonal quad-graphs that may not be derived from a Delaunay triangulation. It is worth noting that the framework established by Bobenko and Skopenkov accommodates scenarios where certain quadrilaterals, denoted by $Q$, degenerate into lines (indicated by $\rho_Q = 0$) or exhibit negative orientation (indicated by $\rho_Q < 0$). These cases arise when the sum of the two angles opposite an edge is equal to or greater than $\pi$, respectively.

It is important to emphasize that while any discrete complex structure with real $\rho_Q$ on a quad-graph $\Lambda$ can be realized through the discretization process outlined earlier for the Riemann surface $\Sigma$, this is not generally applicable when not all $\rho_Q$ are real \cite{BoG17}.


\subsection{Discrete holomorphicity and discrete differentials}\label{sec:differential}

The concept of discrete holomorphicity on quad-graphs is primarily attributed to Mercat \cite{Me01, Me07, Me08}.

\begin{definition}
We define a function $f: V(\Lambda) \rightarrow \mathbb{C}$ to be \textit{discrete holomorphic} if it satisfies the equation \[f(w_+)-f(w_-)=i\rho_Q (f(b_+)-f(b_-))\] for all quadrilaterals $Q \in F(\Lambda)$ with vertices $b_-, w_-, b_+, w_+$ listed in counterclockwise order, starting with a black vertex.
\end{definition}

To establish a discrete analogue of exterior calculus, analogous to the continuous setting, the medial graph plays a significant role as the underlying structure for discrete differentials \cite{BoG17}. In Figure~\ref{fig:medial}, the gray-colored vertices represent the medial graph.

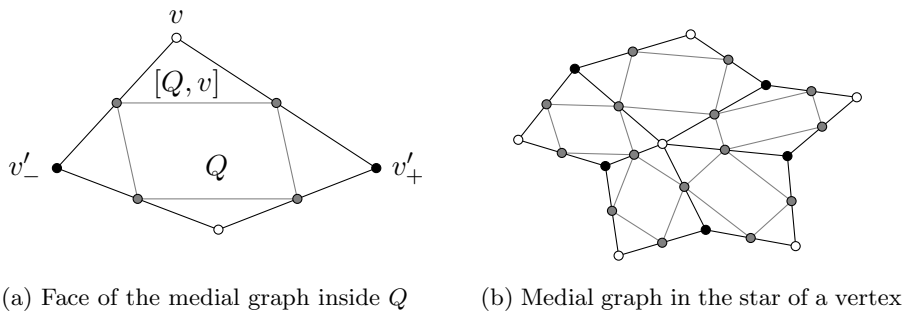
\begin{figure}[htbp]
   \centering
    \subfloat[Face of the medial graph inside $Q$]{
    \beginpgfgraphicnamed{medial1}
			\begin{tikzpicture}[white/.style={circle,draw=black,fill=white,thin,inner sep=0pt,minimum size=1.2mm},
black/.style={circle,draw=black,fill=black,thin,inner sep=0pt,minimum size=1.2mm},
gray/.style={circle,draw=black,fill=gray,thin,inner sep=0pt,minimum size=1.2mm},scale=0.6]
			\clip(-1.7,-6.5) rectangle (7.4,-0.2);
			\draw (-0.6,-4.16)-- (2.02,-1.28);
			\draw (2.02,-1.28)-- (6.4,-4.16);
			\draw (6.4,-4.16)-- (2.94,-5.52);
			\draw (2.94,-5.52)-- (-0.6,-4.16);
			\draw [color=gray] (1.17,-4.84)-- (0.71,-2.72);
			\draw [color=gray] (0.71,-2.72)-- (4.21,-2.72);
			\draw [color=gray] (4.21,-2.72)-- (4.67,-4.84);
			\draw [color=gray] (4.67,-4.84)-- (1.17,-4.84);
			\node[white] (w1) at (2.94,-5.52) {};
			\node[white] (w2) [label=above:$v$] at (2.02,-1.28) {};
			\node[black] (b1) [label=left:$v'_{-}$] at (-0.6,-4.16) {};
			\node[black] (b2) [label=right:$v'_{+}$] at (6.4,-4.16) {};
			\draw (2.9,-4.16) node {$Q$};
			\node[gray] (m1) at (0.71,-2.72) {};
			\node[gray] (m2) at (4.21,-2.72) {};
			\node[gray] (m3) at (4.67,-4.84) {};
			\node[gray] (m4) at (1.17,-4.84) {};
			\draw (2.25,-2.3) node {$[Q,v]$};
		\end{tikzpicture}
		\endpgfgraphicnamed}
		\qquad
		\subfloat[Medial graph in the star of a vertex]{
		\beginpgfgraphicnamed{medial2}
		\begin{tikzpicture}
		[white/.style={circle,draw=black,fill=white,thin,inner sep=0pt,minimum size=1.2mm},
black/.style={circle,draw=black,fill=black,thin,inner sep=0pt,minimum size=1.2mm},
gray/.style={circle,draw=black,fill=gray,thin,inner sep=0pt,minimum size=1.2mm},scale=0.6]
	\clip(2,-4) rectangle (10.8,2);
				\draw (2.4,-1.03)-- (3.64,0.54);
				\draw (3.64,0.54)-- (6.18,1.3);
				\draw (3.64,0.54)-- (5.56,-1.12);
				\draw (5.56,-1.12)-- (4.31,-1.61);
				\draw (4.31,-1.61)-- (2.4,-1.03);
				\draw (4.31,-1.61)-- (4.6,-3.59);
				\draw (9.82,-0.09)-- (7.83,0.18);
				\draw (9.82,-0.09)-- (8.3,-1.39);
				\draw (7.83,0.18)-- (5.56,-1.12);
				\draw (8.3,-1.39)-- (5.56,-1.12);
				\draw (7.83,0.18)-- (6.18,1.3);
				\draw (5.56,-1.12)-- (6.51,-3.02);
				\draw (6.51,-3.02)-- (4.6,-3.59);
				\draw (6.51,-3.02)-- (8.48,-3.38);
				\draw (8.48,-3.38)-- (8.3,-1.39);
				\draw [color=gray] (4.94,-1.36)-- (4.6,-0.29);
				\draw [color=gray] (4.6,-0.29)-- (6.7,-0.47);
				\draw [color=gray] (6.7,-0.47)-- (6.93,-1.25);
				\draw [color=gray] (6.93,-1.25)-- (6.04,-2.07);
				\draw [color=gray] (6.04,-2.07)-- (4.94,-1.36);
					\node[white] (w1) at (9.82,-0.09) {};
					\node[white] (w2) at (5.56,-1.12) {};
					\node[white] (w3) at (2.4,-1.03) {};
					\node[white] (w4) at (8.48,-3.38) {};
					\node[white] (w5) at (4.6,-3.59) {};
					\node[white] (w6) at (6.18,1.3) {};
					\node[black] (b1) at (7.83,0.18) {};
					\node[black] (b2) at (8.3,-1.39) {};
					\node[black] (b3) at (3.64,0.54) {};
					\node[black] (b4) at (6.51,-3.02) {};
					\node[black] (b5) at (4.31,-1.61) {};
					\draw [color=gray] (7.495,-3.2) --(8.39,-2.385)--(6.93,-1.25)--(9.06,-0.74)--(8.825,0.045)--(6.7,-0.47)--(7.005,0.74) --(4.91,0.92)--(4.6,-0.29)--(3.02,-0.245)--(3.355,-1.32)--(4.94,-1.36)--(4.455,-2.6)--(5.555,-3.305)--(6.04,-2.07)--(7.495,-3.2);				
					\node[gray] (m1) at (6.93,-1.25) {};
					\node[gray] (m2) at (6.7,-0.47) {};
					\node[gray] (m3) at (4.6,-0.29) {};
					\node[gray] (m4) at (4.94,-1.36) {};
					\node[gray] (m5) at (6.04,-2.07) {};
					\node[gray] (m6) at (7.495,-3.2) {};
					\node[gray] (m7) at (8.39,-2.385) {};
					\node[gray] (m8) at (9.06,-0.74) {};
					\node[gray] (m9) at (8.825,0.045) {};
					\node[gray] (m10) at (7.005,0.74) {};
					\node[gray] (m11) at (4.91,0.92) {};
					\node[gray] (m12) at (3.02,-0.245) {};
					\node[gray] (m13) at (3.355,-1.32) {};
					\node[gray] (m14) at (4.455,-2.6) {};
					\node[gray] (m15) at (5.555,-3.305) {};
			\end{tikzpicture}
		\endpgfgraphicnamed}
   \caption[]{Medial graph $X$ and notation of its edges.}
   \label{fig:medial}
\end{figure}

\begin{definition}
The \textit{medial graph} $X$ of $\Lambda$ is defined as follows: The vertices of $X$ correspond to the midpoints of the edges in $\Lambda$, and an edge $[Q, v]$ connects the midpoints of two edges that share a vertex $v \in V(\Lambda)$ and belong to the same quadrilateral $Q \in F(\Lambda)$. The faces of $X$, denoted by $F(X)$, are in bijection with $V(\Lambda) \cup F(\Lambda)$.
\end{definition}

Each face $F \in F(X)$ associated with a quadrilateral $Q \in F(\Lambda)$ forms a parallelogram when $Q$ is embedded in the complex plane. The edges of $F$ are parallel to the black and white diagonals of $Q$.

Now, we proceed with the definitions of discrete differentials and discrete two-forms:

A \textit{discrete differential} $\omega$ is a complex-valued function defined on the oriented edges of $X$. It satisfies the property that $\omega(-e) = \omega(e)$ for differently oriented edges $\pm e$. The evaluation of $\omega$ along directed paths or cycles on $X$ is denoted as $\int_P \omega$ and $\oint_P \omega$, respectively. Furthermore, a discrete differential $\omega$ is said to be of \textit{type} $\Diamond$ if, for every quadrilateral $Q \in F(\Lambda)$, the values of $\omega$ on the parallel edges of the corresponding face $F_Q \in F(X)$ are the same. Discrete differentials of type $\Diamond$ play a significant role as they are primarily defined on the oriented edges of the quad-graph $\Lambda$ and its dual graph.

A \textit{discrete two-form} $\Omega$ is a complex-valued function defined on the faces of $X$. The evaluation of $\Omega$ on a set of faces is denoted as $\iint_S \Omega$.

The medial graph approach not only allows us to define a natural product between functions defined on either $V(\Lambda)$ or $F(\Lambda)$ and discrete two-forms with matching support, but also enables the definition of a natural product between functions and discrete one-forms. Specifically, the value of a function at a vertex $v$ or a face $Q$ is multiplied by the value of the discrete one-form on all edges $[Q,v]$.

In a chart $z$ associated with a quadrilateral $Q \in F(\Lambda)$ or the star of a vertex $v \in V(\Lambda)$, we can canonically define $dz$ and $d\bar{z}$. These differentials are of type $\Diamond$. Any discrete differential of type $\Diamond$ can be locally represented as $pdz + qd\bar{z}$, where $p$ and $q$ are complex functions on $F(\Lambda)$ that depend on the chosen chart.

To define the wedge product $dz \wedge d\bar{z}$ in a chart $z = z_Q$, we set $\iint_{F_Q} dz \wedge d\bar{z} = -4i \text{area}(z(Q))$. Similarly, in a chart $z = z_v$, we define $dz \wedge d\bar{z}$ by $\iint_{F_v} dz \wedge d\bar{z} = -4i \text{area}(z(F_v))$, while $\iint_{F_Q} dz \wedge d\bar{z} = 0$ for faces $Q$ incident to $v$. The additional factor of 2 compared to the continuous setting accounts for the fact that the parallelograms in the medial graph cover exactly half of the surface area.

\begin{definition}
Let $f:V(\Lambda) \to \mathbb{C}$ and $h:F(\Lambda) \to \mathbb{C}$. Consider the discrete charts $z_Q$ and $z_v$ associated with $Q \in F(\Lambda)$ and $v\in V(\Lambda)$, respectively, and let $F_Q$ and $F_v$ denote the corresponding faces of $X$ with counterclockwise orientations of their boundaries. We define the \textit{discrete derivatives} $\partial_\Lambda f$, $\bar{\partial}\Lambda f$, $\partial\Diamond h$, and $\bar{\partial}\Diamond h$ as follows:
\begin{align*}
\partial_\Lambda f(Q)&:=\frac{1}{\iint_{F_Q} dz_Q \wedge d\bar{z}Q}\oint_{\partial F_Q} f d\bar{z}Q, \quad \bar{\partial}_\Lambda f (Q):=\frac{-1}{\iint_{F_Q} dz_Q \wedge d\bar{z}Q}\oint_{\partial F_Q} f dz_Q,\\
\partial_\Diamond h(v)&:=\frac{1}{\iint_{F_v} dz_v \wedge d\bar{z}v}\oint_{\partial F_v} h d\bar{z}v, \quad \bar{\partial}_\Diamond h(v):=\frac{-1}{\iint_{F_v} dz_v \wedge d\bar{z}v}\oint_{\partial F_v} h dz_v.
\end{align*}
\end{definition}

In the continuous setting, the formulas provided above serve as first-order approximations of derivatives \cite{ChSm11}. In the discrete context, a function $f:V(\Lambda) \to \mathbb{C}$ is discrete holomorphic if and only if $\bar{\partial}\Lambda f \equiv 0$ in any chart. However, the equation $\bar{\partial}\Diamond h = 0$ is typically dependent on the chosen chart, making the notion of discrete holomorphicity for functions $h:F(\Lambda)\to\mathbb{C}$ not well-defined on general Riemann surfaces.

The introduced concepts of discrete derivatives for functions, discrete one-forms $dz$, $d\bar{z}$, and the discrete two-form $dz \wedge d\bar{z}$ allow us to directly adapt the continuous formulas of the exterior derivative, wedge product, Hodge star, scalar product, formal adjoint $\delta:=-\star d \star$, Laplacian, harmonicity of functions and one-forms, and holomorphicity of one-forms to the discrete setting.

For instance, the discrete exterior derivative $df$ of a complex function $f:V(\Lambda) \to \mathbb{C}$ is given locally by $df = \partial_\Lambda f dz + \bar{\partial}_\Lambda f d\bar{z}$. In the case of the discrete wedge product and discrete Hodge star, we need to consider discrete one-forms of type $\Diamond$. The proof of the chart-independence of the discrete formulas corresponds to the previous definitions by Mercat in \cite{Me01,Me08}. For example, the chart-independence of the discrete exterior derivative corresponds to the \textit{discrete Stokes' theorem} $\iint_F d\omega = \oint_{\partial F} \omega$.


\subsection{Periods and discrete period matrices}\label{sec:period}

We consider a canonical homology basis ${a_1, \ldots, a_g, b_1, \ldots, b_g}$ of $H_1(\Sigma,\mathbb{Z})$. Let $\alpha_1, \ldots, \alpha_g, \beta_1, \ldots, \beta_g$ be closed paths on $X$ that correspond to the homologies $a_1, \ldots, a_g, b_1, \ldots, b_g$, respectively. These paths share a common base point $x_0 \in V(X)$, which is fixed throughout the discussion.

\begin{figure}[htbp]
\begin{center}
\beginpgfgraphicnamed{medial}
\begin{tikzpicture}
[white/.style={circle,draw=black,fill=black,thin,inner sep=0pt,minimum size=1.2mm},
black/.style={circle,draw=black,fill=white,thin,inner sep=0pt,minimum size=1.2mm},
gray/.style={circle,draw=black,fill=gray,thin,inner sep=0pt,minimum size=1.2mm},scale=1.0]
\node[white] (w1)
at (-2,-2) {};
\node[white] (w2)
 at (0,-2) {};
\node[white] (w3)
 at (2,-2) {};
\node[white] (w4)
 at (-1,-1) {};
\node[white] (w5)
 at (1,-1) {};
\node[white] (w6)
 at (-2,0) {};
\node[white] (w7)
 at (0,0) {};
\node[white] (w8)
 at (2,0) {};
\node[white] (w9)
 at (-1,1) {};
\node[white] (w10)
 at (1,1) {};
\node[white] (w11)
 at (-2,2) {};
\node[white] (w12)
 at (0,2) {};
\node[white] (w13)
 at (2,2) {};

\node[black] (b1)
 at (-1,-2) {};
\node[black] (b2)
 at (1,-2) {};
\node[black] (b3)
 at (-2,-1) {};
\node[black] (b4)
 at (0,-1) {};
\node[black] (b5)
 at (2,-1) {};
\node[black] (b6)
 at (-1,0) {};
\node[black] (b7)
 at (1,0) {};
\node[black] (b8)
 at (-2,1) {};
\node[black] (b9)
 at (0,1) {};
\node[black] (b10)
 at (2,1) {};
\node[black] (b11)
 at (-1,2) {};
\node[black] (b12)
 at (1,2) {};

\node[gray] (m1)
 at (0,-1.5) {};
\node[gray] (m2)
 at (0.5,-1) {};
\node[gray] (m3)
 at (1,-0.5) {};
\node[gray] (m4)
 at (1.5,0) {};
\node[gray] (m5)
 at (1,0.5) {};
\node[gray] (m6)
 at (0.5,1) {};
\node[gray] (m7)
 at (0,1.5) {};
\node[gray] (m8)
 at (-0.5,1) {};
\node[gray] (m9)
 at (-1,0.5) {};
\node[gray] (m10)
 at (-1.5,0) {};
\node[gray] (m11)
 at (-1,-0.5) {};
\node[gray] (m12)
 at (-0.5,-1) {};

\draw[dashed] (w1) -- (b1) -- (w2) -- (b2) -- (w3);
\draw[dashed] (b3) -- (w4) -- (b4) -- (w5) -- (b5);
\draw[dashed] (w6) -- (b6) -- (w7) -- (b7) -- (w8);
\draw[dashed] (b8) -- (w9) -- (b9) -- (w10) -- (b10);
\draw[dashed] (w11) -- (b11) -- (w12) -- (b12) -- (w13);

\draw[dashed] (w1) -- (b3) -- (w6) -- (b8) -- (w11);
\draw[dashed] (b1) -- (w4) -- (b6) -- (w9) -- (b11);
\draw[dashed] (w2) -- (b4) -- (w7) -- (b9) -- (w12);
\draw[dashed] (b2) -- (w5) -- (b7) -- (w10) -- (b12);
\draw[dashed] (w3) -- (b5) -- (w8) -- (b10) -- (w13);

\draw (w2) -- (w5) -- (w8) -- (w10) -- (w12) -- (w9) -- (w6) -- (w4) -- (w2);
\draw (b4) -- (b7) -- (b9) -- (b6) -- (b4);

\draw[color=gray] (m1) -- (m2) -- (m3) -- (m4) -- (m5) -- (m6) -- (m7) -- (m8) -- (m9) -- (m10) -- (m11) -- (m12) -- (m1);

\coordinate[label=center:$W(P)$] (z1)  at (-0.2,-0.3) {};
\coordinate[label=center:$P$] (z2)  at (0.5,-1.25) {};
\coordinate[label=center:$B(P)$] (z3)  at (-0.85,-1.7) {};
\end{tikzpicture}
\endpgfgraphicnamed
\caption{Cycles $P$ on $X$, $B(P)$ on $\Gamma$, and $W(P)$ on $\Gamma^*$.}
\label{fig:contours2}
\end{center}
\end{figure}
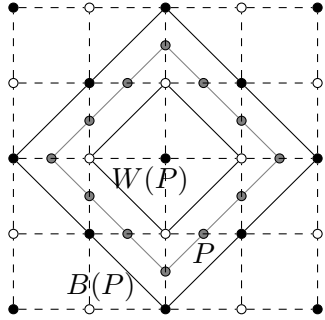

Figure~\ref{fig:contours2} illustrates how each oriented cycle $P$ on $X$ induces closed paths $B(P)$ and $W(P)$ on $\Gamma$ and $\Gamma^*$, respectively, by adding the respective parallel diagonal. We denote the sets of oriented edges of $X$ that are parallel to the edges of $B(P)$ and $W(P)$ as $BP$ and $WP$, respectively.

\begin{definition}
Let $\omega$ be a closed discrete differential of type $\Diamond$. For $1 \leq k \leq g$, we define its $a_k$-periods as $A_k:=\oint_{\alpha_k} \omega$ and $b_k$-periods as $B_k:=\oint_{\beta_k} \omega$. We also define its black $a_k$-periods as $A^B_k:=2\int_{B\alpha_k} \omega$ and its black $b_k$-periods as $B^B_k:=2\int_{B\beta_k} \omega$. Similarly, we define its white $a_k$-periods as $A^W_k:=2\int_{W\alpha_k} \omega$ and its white $b_k$-periods as $B^W_k:=2\int_{W\beta_k} \omega$.
\end{definition}

Clearly, $2A_k=A_k^B+A_k^W$ and $2B_k=B_k^B+B_k^W$. In some cases, it is more convenient to work with the integrals of closed discrete differentials instead. These integrals have a one-to-one correspondence with the discrete differentials and exhibit multi-valued behavior with periods, unless the differential vanishes. We denote the universal cover of $(\Sigma, \Lambda)$ as $(\tilde{\Sigma}, \tilde{\Lambda})$, which is a non-compact discrete Riemann surface. The deck transformation corresponding to the closed cycle $P$ on $\Sigma$ is denoted as $d_P$.

\begin{definition}
A function $f: V(\tilde{\Lambda}) \to \mathbb{C}$ is said to be \textit{multi-valued} with black periods $A_1^B, A_2^B, \ldots, A_g^B$, $B_1^B, B_2^B, \ldots, B_g^B \in \mathbb{C}$ and white periods $A_1^W, A_2^W, \ldots, A_g^W$, $B_1^W, B_2^W, \ldots, B_g^W \in \mathbb{C}$ if it satisfies the following conditions:
\begin{align*}
f(d_{\alpha_k}b)=f(b)+A_k^B, \quad f(d_{\alpha_k}w)=f(w)+A_k^W, \quad f(d_{\beta_k}b)=f(b)+B_k^B, \quad f(d_{\beta_k}w)=f(w)+B_k^W
\end{align*}
for any $1 \leq k \leq g$, each black vertex $b \in V(\tilde{\Gamma})$, and each white vertex $w \in V(\tilde{\Gamma}^*)$.
\end{definition}

The periods of closed discrete differentials satisfy a discrete analogue of the classical Riemann bilinear identity. Additionally, for any given set of discrete periods, there exists a unique discrete harmonic differential with these periods. We provide a proof of this theorem in \cite{BoG17}, slightly modified below.

\begin{theorem}\label{th:harmonic_existence}
Let $A_k^B, B_k^B, A_k^W, B_k^W$, $1 \leq k \leq g$, be $4g$ given complex numbers. Then, there exists a unique discrete harmonic differential $\omega$ with these black and white periods.
\end{theorem}

\begin{proof}
We will prove the statement for real periods. Consider the vector space $R$ of all multi-valued functions $f: V(\tilde{\Lambda}) \to \mathbb{R}$ that have the given black and white periods. For any such function $f$, $df$ is a discrete differential of type $\Diamond$ on $(\Sigma, \Lambda)$. The discrete Dirichlet energy $E_\Diamond(f) := \langle df, df \rangle$ is convex, quadratic, and nonnegative \cite{BoG17}, which implies the existence of a minimum $M_0: V(\tilde{\Lambda}) \to \mathbb{R}$.

Any element $f \in \mR$ can be written as $f = M_0 + f'$, where $f': V(\Lambda) \to \mathbb{R}$. For $\lambda \in \mathbb{R}$, we have \[E_\Diamond(f)=E_\Diamond(_0)+2\lambda \langle dM_0, df'\rangle+\lambda^2 E_\Diamond(f').\] Since $M_0$ is the minimizer, it follows that $\langle dM_0, df' \rangle = 0$ for all functions $f': V(\Lambda) \to \mathbb{R}$.

Let $\omega := df$. The discrete differential $\omega$ is of type $\Diamond$ and has the required periods. Since $\delta$ is the formal adjoint of $d$, we have \[0=\langle \omega, df' \rangle = \langle \delta \omega, f' \rangle = \langle \delta dM_0, f' \rangle = - \langle \triangle M_0, f' \rangle, \] where $\triangle M_0$ is well-defined on $(\Sigma, \Lambda)$. Therefore, we have $\delta \omega = 0 = d\omega$ and $\triangle M_0 = 0$, indicating that $M_0$ and $\omega$ are discrete harmonic. The uniqueness of $\omega$ follows from discrete Liouville's Theorem.
\end{proof}

\begin{corollary}\label{cor:harmonic}
If $\omega$ is a discrete harmonic differential, then it has the minimal discrete Dirichlet energy $\langle \omega, \omega \rangle$ among all discrete one-forms with the same black and white periods. Furthermore, $\langle \omega, df \rangle = 0$ for all $f: V(\Lambda) \to \mathbb{C}$.
\end{corollary}

An immediate consequence of Theorem~\ref{th:harmonic_existence} is that the complex vector space of discrete holomorphic differentials has dimension $2g$. Discrete periods behave similarly to their continuous counterparts, and the bipartiteness of $\Lambda$ leads to the splitting into black and white periods. This implies the following:

\begin{theorem}\label{th:holomorphic_existence}
For any $2g$ complex numbers $A_k^B, A_k^W$, $1 \leq k \leq g$, there exists exactly one discrete holomorphic differential $\omega$ with these black and white $a$-periods.

For any $4g$ real numbers $\text{Re}(A_k^B), \text{Re}(B_k^B), \text{Re}(A_k^W), \text{Re}(B_k^W)$, there exists exactly one discrete holomorphic differential $\omega$ such that its black and white periods have these real parts.
\end{theorem}

\begin{definition}
Let $\omega_k^B$, $1 \leq k \leq g$, be the unique discrete holomorphic differential with a black $a_j$-period $\delta_{jk}$ and vanishing white $a$-periods. Let $\omega_k^W$, $1 \leq k \leq g$, be the unique discrete holomorphic differential with a white $a_j$-period $\delta_{jk}$ and vanishing black $a$-periods. We refer to the basis of these $2g$ discrete differentials as the \textit{canonical basis of discrete holomorphic differentials}.

We define the $(g \times g)$-matrices $\Pi^{B,B}, \Pi^{W,B}, \Pi^{B,W}, \Pi^{W,W}$ with entries
\begin{align*}\Pi^{B,B}_{jk}:=2\int_{Bb_j}\omega^B_k, \quad \Pi^{W,B}_{jk}:=2\int_{Wb_j}\omega^B_k,\quad \Pi^{B,W}_{jk}:=2\int_{Bb_j}\omega^W_k, \quad \Pi^{W,W}_{jk}:=2\int_{Wb_j}\omega^W_k.
\end{align*} 

The \textit{complete discrete period matrix} is the $(2g\times 2g)$-matrix defined by \[\tilde{\Pi}:=\left( \begin{matrix} \Pi^{B,W} & \Pi^{B,B}\\ \Pi^{W,W} & \Pi^{W,B}\end{matrix}\right).\]
\end{definition}

We also define the continuous period matrix $\Pi_{\Sigma}$ of $\Sigma$ using the same basis of homology. However, it is the discrete period matrix $\Pi = (\Pi^{B,W} + \Pi^{B,B} + \Pi^{W,W} + \Pi^{W,B})/2$ that we mainly compare with $\Pi_{\Sigma}$. $\Pi$ can be directly defined using discrete differentials with equal black and white $a$-periods.

\begin{definition}
The unique set of $g$ discrete holomorphic differentials $\omega_k$ satisfying $2\int_{Ba_j} \omega_k =\delta_{jk}$ and $2\int_{Wa_j} \omega_k = \delta_{jk}$ for all $1 \leq j,k \leq g$ is called the \textit{canonical set}. The $(g \times g)$-matrix $\left(\Pi_{jk}\right)_{j,k=1}^g$ with entries $\Pi_{jk} = \int_{b_j} \omega_k$ is the \textit{discrete period matrix} of the discrete Riemann surface $(\Sigma,\Lambda,z)$.
\end{definition}

\begin{example}
Let us consider a discretization of a flat torus $\Sigma = \mathbb{C}/(\mathbb{Z} + \mathbb{Z}\tau)$ of modulus $\tau \in \mathbb{C}$ with $\text{Im}(\tau) > 0$. In this case, the continuous period of $\Sigma$ is $\tau$. On the medial graph $X$, the differential $dz$ is globally defined and discrete holomorphic. Thus, $\Pi$ is equal to the $b$-period of $dz$, which is also $\tau$. Note that the black and white $b$-periods do not have to coincide if the $a$-periods do not coincide \cite{BoSk12, BoG17}.
\end{example}

\begin{theorem}\label{th:period_matrix}
The matrices $\Pi$ and $\tilde{\Pi}$ are symmetric, and their imaginary parts are positive definite.
\end{theorem}

Since the block matrices in the diagonal of a symmetric positive definite matrix are themselves symmetric and positive definite, we also have the following corollary:

\begin{corollary}\label{cor:period_submatrix}
The imaginary parts of the symmetric matrices $\Pi^{W,B}$ and $\Pi^{B,W}$ are positive definite.
\end{corollary}

\begin{lemma}\label{lem:period_orthodiagonal}
Assuming all quadrilaterals of $(\Sigma,\Lambda)$ are orthodiagonal (i.e., $\rho_Q$ is real for all $Q \in F(\Lambda)$), we have that $\Pi^{W,B}$ and $\Pi^{B,W}$ are purely imaginary, while $\Pi^{B,B}$ and $\Pi^{W,W}$ are real matrices.
\end{lemma}

\begin{proof}
In \cite{BoG17}, we established a connection between our definition of the discrete Hodge star on discrete one-forms and Mercat's definition given in \cite{Me08}. Specifically, if $\omega$ is a discrete one-form of type $\Diamond$ defined on the oriented edges of the boundary of the face of $X$ corresponding to $Q$, then the equations governing the discrete Hodge star in a discrete chart $z=z_Q$ are given by:
\begin{align*}
\int_e \star\omega = \cot\left(\phi_Q\right) \int_e \omega - \frac{|z(e)|}{|z(e^*)| \sin\left(\phi_Q\right)}\int_{e^*}\omega \quad \text {and} \quad
\int_{e^*} \star\omega = \frac{|z(e^*)|}{|z(e)| \sin\left(\varphi_Q\right)} \int_e \omega - \cot\left(\phi_Q\right)\int_{e^*}\omega,
\end{align*}
where $\phi_Q$ is the angle under which the diagonals of $Q$ intersect, and $e$ and $e^*$ are oriented edges parallel to the black and white diagonal of $Q$, respectively, with $\im \left(e^*/e\right) > 0$.

For orthodiagonal quadrilaterals where $\rho_Q$ is real, $\phi_Q = \pi/2$, simplifying the equations to:
\begin{align*}
\int_e \star\omega = -\frac{|z(e)|}{|z(e^*)|}\int_{e^*}\omega \quad \text{and} \quad \int_{e^*} \star\omega = \frac{|z(e^*)|}{|z(e)|} \int_e \omega.
\end{align*}
In particular, if $\omega$ is a discrete one-form that is zero on white edges, then $\star\omega$ is zero on black edges, and vice versa. Additionally, if $\omega$ is real, then $\star\omega$ is also real.

Consider a discrete holomorphic differential $\omega$ and let $\alpha$ be the discrete one-form that is zero on white edges and agrees with $\omega$ on black edges. We define $\beta:=\omega-\alpha$, which is zero on black edges. Both $\alpha$ and $\beta$ are of type $\Diamond$. From $\omega=-i\star\omega$, we have $\star\alpha+\star\beta=-i\alpha-i\beta$. Since only $\star\alpha$ and $-i\beta$ are non-zero on white edges, we find that $\beta=i\star\alpha$.

We have $d\alpha+id\star\alpha=d\omega=0$. Given that $\alpha$ is zero on white edges and $\star\alpha$ on black edges, we can apply discrete Stokes' theorem to conclude that $d\alpha$ is zero on faces of the medial graph corresponding to white vertices, and $d\star\alpha$ is zero on faces corresponding to black vertices. This implies $d\alpha=d\star\alpha=0$. Thus, $\alpha$ is a discrete harmonic form. Moreover, since $\star \bar{\alpha}=\overline{\star \alpha}$, both $\re \alpha$ and $\im \alpha$ are discrete harmonic forms as well. This implies that $\omega=\varrho+\iota$, where \[\varrho:=\re \alpha + i\star \re \alpha, \quad \iota:=i\im \alpha -\star \im \alpha.\]

The forms $\varrho$ and $\iota$ are both discrete holomorphic. The black periods of $\varrho$ correspond to the real parts of the black periods of $\omega$, while the white periods of $\varrho$ correspond to $i$ times the imaginary parts of the white periods of $\omega$. Similarly, the white periods of $\iota$ correspond to the real parts of the white periods of $\omega$, and the black periods of $\iota$ correspond to $i$ times the imaginary parts of the real periods of $\omega$.

Let $\omega=\omega_k^B$, $1\leq k \leq g$, be the unique discrete holomorphic differential with a black $a_j$-period of $\delta_{jk}$ and vanishing white $a$-periods. According to our previous considerations, $\varrho=\varrho_k^B$ is a discrete holomorphic differential with the same black and white $a$-periods as $\omega_k^B$. Therefore, by Theorem \ref{th:holomorphic_existence}, we have \[\omega_k^B=\varrho_k^B=\re \alpha_k^B+i\star \re \alpha_k^B.\]

Hence, the black $b$-periods of $\omega_k^B$ are real, and its white $b$-periods are purely imaginary. We can conclude that $\re \Pi^{W,B}=0=\im \Pi^{B,B}$. Similarly, we have $\re \Pi^{B,W}=0=\im \Pi^{W,W}$.
\end{proof}


\subsection{Discrete Dirichlet energy as a quadratic form}\label{sec:Dirichlet}

In Section~\ref{sec:convergence1}, we will use the same underlying idea as Bobenko and Skopenkov in \cite{BoSk12} to analyze the convergence of the discrete period matrix. We deduce the convergence of the discrete period matrix from the convergence of the discrete Dirichlet energy of discrete harmonic differentials with given periods (or equivalently, of discrete harmonic functions on the universal cover $(\tilde{\Sigma},\tilde{\Lambda})$ with given periods). To do this, we represent both the continuous and discrete Dirichlet energy of continuous and discrete harmonic differentials as a quadratic form in the vector of real $a$- and $b$-periods.

Before we continue, let us review the definition of the discrete gradient by Skopenkov in \cite{Sk13} and relate it to the discrete differential.

\begin{definition}
Let $Q \in F(\Lambda)$ and $f$ be a real function on its vertices. In a chart $z=z_Q$ of $Q$, we define the \textit{discrete gradient} $\nabla_Q f$ as the unique vector that satisfies, for all edges $e$ of $F_Q$: \[\nabla_Q f \cdot z(e)=\int_e df.\]
\end{definition}
\begin{remark}
The discrete gradient of a multi-valued function on $V(\tilde{\Lambda})$ is a well-defined function on $F(\Lambda)$.
\end{remark}

In the case of a smooth function $F: U \subset \mathbb{C} \to \mathbb{R}$, an elementary calculation shows that $4|\partial F| = |\nabla F|$, such that $\langle dF, dF \rangle = \int |\nabla F|^2$. The same is true in the discrete setting.

\begin{lemma}\label{lem:gradient_differential}
Let $Q \in F(\Lambda)$ and $f$ be a real function on its vertices. In a chart $z = z_Q$ of $Q$, we have \[\iint_{F(Q)} df \wedge \star d\bar{f}= \left\|\nabla_Q f\right\|^2 \textnormal{area}(Q).\]
\end{lemma}

\begin{proof}
Using $df=\partial_\Lambda fdz + \bar{\partial}_\Lambda fd\bar{z}$, we compute \[\iint_{F(Q)} df \wedge \star d\bar{f}=4 |\partial_\Lambda f|^2 \textnormal{area}(Q).\] On the other hand, let us integrate $\nabla_Q f$ over $Q$. We obtain a linear function $F$ on $Q$ that we restrict to a function $f'$ on the vertices of $Q$. By definition of the discrete gradient, we have $df = df'$. We have seen in our previous work \cite{BoG15} that $\partial\Lambda f'$ agrees with the smooth derivative. In particular, \[4|\partial_\Lambda f|=4|\partial_\Lambda f'|=4|\partial F|=|\nabla F|=|\nabla_Q f|. \qedhere\]
\end{proof}

We introduce the scalar product and norm associated with the discrete Dirichlet energy for real multi-valued functions on the discrete Riemann surface $(\Sigma, \Lambda)$. Using Lemma~\ref{lem:gradient_differential}, we define the positive definite scalar product and norm as follows:
\begin{align*}
a(f, g) &:= \langle df, dg \rangle = \sum\limits_{Q\in F(\Lambda)} \int_Q \nabla_Q f \cdot \nabla_Q g,\\
\left|f\right| &:= \sqrt{a(f,f)}.
\end{align*}
In the case of smooth multi-valued functions $F, G: \tilde{\Sigma} \to \mathbb{R}$, we can also define $a(F,G)$ by replacing the discrete gradient with the smooth gradient. The norm $\left|F\right|^2$ corresponds to the Dirichlet energy of $F$. Additionally, $a(f,F)$ is well-defined.

It is worth noting that we omit the dependence of the scalar product $a$ on $\Lambda$ since it is already encoded in the discrete functions. Furthermore, $a$ is derived from an inner product of the gradients, ensuring that the Cauchy-Schwarz inequality holds true.

Next, we state the representation lemma in the continuous setting, as presented in \cite{BoSk12} (with the off-diagonal blocks interchanged). The proof of this lemma follows a similar structure to the proof of the corresponding discrete lemma, which we will present afterwards.

\begin{lemma}\label{lem:energy_continuous}
Let $\Omega$ denote the unique holomorphic differential on $\Sigma$ with real parts of its $a$- and $b$-periods given by $\mathcal{A}_1,\ldots,\mathcal{A}_g$ and $\mathcal{B}_1,\ldots,\mathcal{B}_g$, respectively. The corresponding smooth real harmonic function with these periods can be expressed as $U=\int \re(\Omega)$, up to a choice of integration constant. We consider the vector $P=(\mathcal{A}_1,\ldots,\mathcal{A}_g,\mathcal{B}_1,\ldots,\mathcal{B}g)^T \in \mathbb{R}^{2g}$. Then, the Dirichlet energy $\langle \Omega, \Omega \rangle=2 \langle dU, dU \rangle=2\|U\|^2$ can be represented as a quadratic form in terms of the period matrix $\Pi_{\Sigma}$: \[\langle \Omega, \Omega \rangle=P^T E_{\Sigma} P:= P^T \left( \begin{matrix} 2\re \Pi_{\Sigma} \left( \im \Pi_{\Sigma} \right)^{-1} \re \Pi_{\Sigma} + 2\im \Pi_{\Sigma}  &  -2\re \Pi_{\Sigma} \left( \im \Pi_{\Sigma} \right)^{-1}\\ -2\left( \im \Pi_{\Sigma} \right)^{-1} \re \Pi_{\Sigma} & 2\left( \im \Pi_{\Sigma} \right)^{-1} \end{matrix}\right)  P.\]
\end{lemma}

\begin{lemma}\label{lem:energy_discrete}
Let $\omega$ be the unique discrete holomorphic differential on $(\Sigma,\Lambda)$ with real parts of its black $a$- and $b$-periods given by $\mathcal{A}^B_1,\ldots,\mathcal{A}^B_g$ and $\mathcal{B}^B_1,\ldots,\mathcal{B}^B_g$, and real parts of its white $a$- and $b$-periods given by $\mathcal{A}^W_1,\ldots,\mathcal{A}^W_g$ and $\mathcal{B}^W_1,\ldots,\mathcal{B}^W_g$. The corresponding multi-valued discrete real harmonic function $u$ with these periods can be expressed as $du= \re(\omega)$, where we make a choice of two integration constants on $\Gamma$ and $\Gamma^*$. We consider the vector $P'=(\mathcal{A}^W_1,\ldots,\mathcal{A}^W_g,\mathcal{A}^B_1,\ldots,\mathcal{A}^B_g,\mathcal{B}^B_1,\ldots,\mathcal{B}^B_g,\mathcal{B}^W_1,\ldots,\mathcal{B}^W_g)^T \in \mathbb{R}^{4g}$. Then, the discrete Dirichlet energy $\langle \omega, \omega \rangle=2 \langle u, u \rangle=2\|u\|^2$ can be represented as a quadratic form in terms of the complete discrete period matrix $\tilde{\Pi}$: \[\langle \omega, \omega \rangle= P'^T E_{\Lambda} P':= P'^T \left( \begin{matrix} \re \tilde{\Pi} \left( \im \tilde{\Pi} \right)^{-1} \re \tilde{\Pi} + \im \tilde{\Pi}  &  -\re \tilde{\Pi} \left( \im \tilde{\Pi} \right)^{-1}\\ -\left( \im \tilde{\Pi} \right)^{-1} \re \tilde{\Pi} & \left( \im \tilde{\Pi} \right)^{-1} \end{matrix}\right)  P'.\]
\end{lemma}

\begin{proof}
Clearly, $2du=2\re \omega=\omega + \bar{\omega}$. Hence, $4 \langle du, du \rangle=\langle \omega, \omega \rangle + \langle \omega, \bar{\omega} \rangle + \langle \omega, \bar{\omega} \rangle + \langle \bar{\omega}, \bar{\omega} \rangle=2 \langle \omega, \omega \rangle$, which follows from the fact that $\langle \bar{\omega}, \bar{\omega} \rangle = \overline{\langle \omega, \omega \rangle}=\langle \omega, \omega \rangle$ and $\omega \wedge \star \omega = -i \omega \wedge \omega =0$.

Let us consider the black and white $a$- and $b$-periods of $\omega$, denoted by $A_k^B,B_k^B$ and $A_k^W,B_k^W$ respectively, where $1 \leq k \leq g$. Since $\omega$ is a discrete holomorphic differential, we have $\star \omega=-i \omega$. Applying the discrete Riemann bilinear identity to the closed differentials $\omega$ and $\bar{\omega}$, we obtain:
\begin{align*}
\langle \omega,\omega \rangle&= \iint\limits_{F(X)} \omega \wedge \star\bar{\omega}= i\iint\limits_{F(X)} \omega \wedge \bar{\omega}\\
&= \frac{i}{2}\sum_{k=1}^g \left(A_k^B \bar{B}_k^W-B_k^B {\bar{A}}_k^W\right)+\frac{i}{2}\sum_{k=1}^g \left(A_k^W {\bar{B}}_k^B-B_k^W {\bar{A}}_k^B\right)\\
&=\sum_{k=1}^g \left(\re A_k^B \im B_k^W -\im A_k^B \re B_k^W + \re A_k^W \im B_k^B -\im A_k^W \re B_k^B \right).
\end{align*}

The complete discrete period matrix relates the $b$-periods of $\omega$ to its $a$-periods. By rewriting the corresponding system of equations, we can express the imaginary parts $\operatorname{Im} A_k^B, \operatorname{Im} A_k^W, \operatorname{Im} B_k^W, \operatorname{Im} B_k^B$ in terms of their corresponding real parts $\re A_k^B=\mathcal{A}_k^B, \re A_k^W=\mathcal{A}_k^W, \re B_k^W=\mathcal{B}_k^W, \re B_k^B=\mathcal{B}_k^B$.

Consider the canonical basis of discrete holomorphic differentials ${\omega_1^W,\ldots,\omega_g^W,\omega_1^B,\ldots,\omega_g^B}$. We can express $\omega$ in this basis as a linear combination: $\omega = \sum_{k=1}^g \left( A_k^W \omega_k^W+ A_k^B \omega_k^B\right)$. Now, let us introduce the $g$-dimensional vectors $A_W:= (A_1^W,\ldots,A^W_g)^T$, $A_B:=(A_1^B,\ldots,A^W_g)^T$, $B_W:= (B_1^W,\ldots,B^W_g)^T$, $B_B:=(B_1^B,\ldots,B^W_g)^T$, and the $2g$-dimensional vectors $A:=(A_W^T,A_B^T)^T$ and $B:=(B_B^T,B_W^T)^T$. Note that $P'^T=\re (A^T,B^T)$. Thus, we can express $B$ as follows:
\[ B = \left( \begin{matrix} B_B \\ B_W \end{matrix} \right) = \left( \begin{matrix} \Pi^{B,W} A_W + \Pi^{B,B} A_B\\ \Pi^{W,B} A_B + \Pi^{W,W} A_W \end{matrix} \right) = \left( \begin{matrix} \Pi^{B,W} & \Pi^{B,B}\\ \Pi^{W,W} & \Pi^{W,B}\end{matrix}\right) \left( \begin{matrix} A_W \\ A_B \end{matrix} \right)= \tilde{\Pi} A.\]

Now, let us proceed with the calculations. We have:
\begin{align*}
\re B= \re \tilde{\Pi} \re A - \im \tilde{\Pi} \im A &\Rightarrow \im A = \left( \im \tilde{\Pi} \right)^{-1} \re \tilde{\Pi} \re A - \left( \im \tilde{\Pi} \right)^{-1} \re B,\\
\im B= \re \tilde{\Pi} \im A + \im \tilde{\Pi} \re A &\Rightarrow \im B = \left(\re \tilde{\Pi} \left( \im \tilde{\Pi} \right)^{-1} \re \tilde{\Pi} + \im \tilde{\Pi}\right) \re A - \re \tilde{\Pi} \left( \im \tilde{\Pi} \right)^{-1} \re B.
\end{align*}

By substituting these expressions into the equation obtained from the discrete Riemann bilinear identity, we obtain:
\begin{align*}
\langle \omega,\omega \rangle&=\sum_{k=1}^g \left(\re A_k^B \im B_k^W -\im A_k^B \re B_k^W + \re A_k^W \im B_k^B -\im A_k^W \re B_k^B \right)=\re A^T \im B - \re B^T \im A \\
&=\re A^T \left(\re \tilde{\Pi} \left( \im \tilde{\Pi} \right)^{-1} \re \tilde{\Pi} + \im \tilde{\Pi}\right) \re A - \re A^T \re \tilde{\Pi} \left( \im \tilde{\Pi} \right)^{-1} \re B\\
&- \re B^T \left( \im \tilde{\Pi} \right)^{-1} \re \tilde{\Pi} \re A + \re B^T \left( \im \tilde{\Pi} \right)^{-1} \re B\\
&=\left( \re A^T, \re B^T \right) \left( \begin{matrix} \re \tilde{\Pi} \left( \im \tilde{\Pi} \right)^{-1} \re \tilde{\Pi} + \im \tilde{\Pi}  & -\re \tilde{\Pi} \left( \im \tilde{\Pi} \right)^{-1}\\ -\left( \im \tilde{\Pi} \right)^{-1} \re \tilde{\Pi} & \left( \im \tilde{\Pi} \right)^{-1} \end{matrix}\right)  \left( \begin{matrix} \re A \\ \re B \end{matrix}\right) = P'^T E_{\Lambda} P'. \qedhere
\end{align*}
\end{proof}

Note that $E_{\Sigma}$ includes a factor of 2 that is not present in $E_{\Lambda}$ due to the discrete Riemann bilinear identity having a factor of $1/2$, which is absent in the smooth case.


\section{Convergence of the discrete Dirichlet energy}\label{sec:convergence1}

This chapter focuses on proving the convergence of the discrete Dirichlet energy of a discrete harmonic differential, which has prescribed $a$- and $b$-periods (with black and white periods being equal), to the continuous Dirichlet energy of the corresponding smooth harmonic differential with the same prescribed periods. The theorem stating this convergence is presented in Section~\ref{sec:convergence_theorems}, along with a discussion of its implications for the related quadratic form. While these theorems hold true for general discrete Riemann surfaces, the convergence of the corresponding multi-valued discrete harmonic functions to their continuous counterparts is only established in the case of orthodiagonal discrete Riemann surfaces. We will outline how the proof given by Bobenko and Skopenkov in \cite{BoSk12} for Delaunay-Voronoi quadrangulations can be adapted to our more general setting.

Before diving into the proof, we provide some geometric preliminaries in Section~\ref{sec:preliminary}. Specifically, we adapt the concept of $h$-adaptedness introduced in \cite{BoBu17} to our quad-graph framework. This adaptation ensures a linear rate of convergence. In contrast, in the general case, the convergence rate depends on the orders of conical singularities. Section~\ref{sec:notation} covers the basic notation and outlines our proof strategy. We draw inspiration from nonconforming finite elements \cite{Br07} and base our proof on the Second Lemma of Strang. We discuss the approximation and consistency errors separately in Sections~\ref{sec:approximation} and~\ref{sec:consistency}.


\subsection{Geometric preliminaries}\label{sec:preliminary}

To establish the convergence of the discrete Dirichlet energy and the discrete period matrix, it is essential to have bounded interior angles and intersection angles of diagonals in quadrilaterals. We introduce the concept of $\phi$\textit{-regularity} to ensure this boundedness.

\begin{definition}
We define a discrete Riemann surface $(\Sigma,\Lambda)$ to be $\phi$-regular if the interior angles of all quadrilaterals $Q \in F(\Lambda)$ are bounded from below by $\phi$, and the intersection angle between the diagonals of each quadrilateral lies in the interval $[\phi,\pi-\phi]$. In other words, we require that $|\arg(\rho_Q)|\leq \frac{\pi}{2}-\phi$, where $\arg(\rho_Q)\in(-\pi,\pi]$ denotes the argument of the complex number $\rho_Q$.
\end{definition}

Throughout our analysis, we will assume that $(\Sigma,\Lambda)$ is $\phi$-regular.

In the upcoming sections, we will observe that the presence of conical singularities with a total angle of $4\pi$ or more leads to a slower rate of convergence. In order to achieve the same linear convergence rate $O(h)$ as when such singularities do not exist, we need to consider \textit{adapted} cell decompositions $\Lambda$. This notion is inspired by the work of Bobenko and B\"ucking in \cite{BoBu17} for triangulations.

\begin{definition}
We say that a discrete Riemann surface $(\Sigma,\Lambda)$ is $h$-adapted if its maximum edge length is $h$. Moreover, for any conical singularity $O \in S$ with $\gamma_O \leq 1/2$, we require that $|g_O(x)-g_O(y)|\leq h$ for any edge $xy \in E(\Lambda) \cap D_O$ within the chart $(D_O,g_O)$ defined in Section~\ref{sec:Riemann}.
\end{definition}

Let $xy \in E(\Lambda) \cap D_O$ be an edge, where $x$ and $y$ are points on the discrete Riemann surface $(\Sigma, \Lambda)$. Assuming that $x$ is closer to the conical singularity $O$ than $y$, i.e., $|Ox|\leq |Oy|$, we address a correction to the claim made in \cite{BoBu17} regarding $h$-adaptedness and the inequality $|xy|\leq h|Ox|^{1-\gamma_O}$. The presence of the mapping function $g_O$, which scales the angle $\angle xOy$ by the factor $\gamma_O\leq 1/2$, complicates the relationship between $|xy|$ and $|g_O(x)-g_O(y)|$.

If $\angle xOy$ is small and $|Oy|=|Ox|<1$ approaches 1, we observe that $|g_O(x)-g_O(y)|\approx \gamma_O |xy|$ due to the approximation $\sin(\alpha)\approx \alpha$ for small $\alpha>0$. As a result, if we impose the condition $|g_O(x)-g_O(y)|=h$, it would imply $|xy| \approx h\gamma_O^{-1}> h|Ox|^{1-\gamma_O}$. Therefore, we present a corrected version of the claim.

\begin{lemma}\label{lem:adapted}
Let $\gamma_O \leq 1/2$ and consider an edge $xy \in E(\Lambda) \cap D_O$ on the discrete Riemann surface $(\Sigma, \Lambda)$, such that $|Ox|\leq |Oy|$. Then, we have the inequality: \[|xy|\leq \left(1+\frac{\pi}{2\gamma_O}\right)h|Ox|^{1-\gamma_O}.\]
\end{lemma}

\begin{proof}
In polar coordinates $(r,\psi)$, we have $g_O(r,\psi)=r^{\gamma_O} \exp(i\gamma_O\psi)$. We will estimate the radial and angular components of $|xy|$ separately and use the triangle inequality to derive a bound for $|xy|$.

For the radial component, we have:
\begin{align*}
h &\geq |g_O(x)-g_O(y)|\geq |Oy|^{\gamma_O}-|Ox|^{\gamma_O} \geq |Oy|\cdot |Ox|^{\gamma_O-1}-|Ox|^{\gamma_O}\\
\Rightarrow h |Ox|^{1-\gamma_O} &\geq |Oy|-|Ox|
\end{align*}

Next, we consider the angle component. Since $\gamma_O \leq 1/2$, we have $\gamma_O\alpha:=\gamma_O|\angle xOy|\leq \pi/2$. Thus, $\sin(\gamma_O\alpha/2)\geq \gamma_O\alpha/\pi$. Since $|Ox|\leq |Oy|$, the edge $g_O(x)g_O(y)$ is not smaller than the chord corresponding to the angle $\gamma_O\alpha$ and the radius $|Ox|^{\gamma_O}$. Therefore, we have:
\begin{align*}
h &\geq |g_O(x)-g_O(y)|\geq 2 |Ox|^{\gamma_O} \sin\left(\frac{\gamma_O\alpha}{2}\right) \geq 2 |Ox|^{\gamma_O} \frac{\gamma_O\alpha}{\pi}\\
\Rightarrow 2 |Ox| \sin\left(\frac{\alpha}{2}\right) &\leq |Ox| \alpha \leq |Ox|^{1-\gamma_O}h\frac{\pi}{2\gamma_O}.
\end{align*}

Combining the bounds for the radial and angular components, we apply the triangle inequality to obtain: \[|xy|\leq (|Oy|-|Ox|) + 2 |Ox| \sin\left(\frac{\alpha}{2}\right)\leq \left(1+\frac{\pi}{2\gamma_O}\right)h|Ox|^{1-\gamma_O}. \qedhere\]
\end{proof}

In our proof, we will primarily focus on the case of general discrete Riemann surfaces, and provide brief comments on the relatively minor modifications required for the case of $h$-adaptedness.


\subsection{Notation and strategy of proof}\label{sec:notation}

We introduce some notation for the rest of this chapter. Let $h$ denote the maximum edge length of a discrete Riemann surface $(\Sigma,\Lambda)$. We will always assume that $(\Sigma,\Lambda)$ is $\phi$-regular.

Consider a fixed vector of periods $P \in \mathbb{R}^{2g}$. We define $U$ to be a multi-valued smooth harmonic function with its $a$- and $b$-periods given by $P$. Similarly, $u$ is a multi-valued discrete harmonic function with its black and white $a$- and $b$-periods equal to $P$, where the corresponding black and white periods coincide. $U_\Lambda$ represents the restriction of $U$ to $V(\tilde{\Lambda})$, and it is a multi-valued discrete function with equal black and white periods that match those of $u$.

Let $M_P$ be the set of all multi-valued functions $f:V(\tilde{\Lambda})\to\mathbb{R}$ with the same periods as $u$, and let $M_0$ be the set of all functions $f:V(\Lambda)\to\mathbb{R}$.

In their work \cite{BoSk12}, Bobenko and Skopenkov considered triangulations $\Gamma$ of $\Sigma$ and the dual lattice $\Gamma^*$ consisting of the circumcenters of triangles. The quadrilateral cell decomposition $\Lambda$ was defined by the black vertices $V(\Gamma)$ and the white vertices $V(\Gamma^*)$, where the vertices of a triangle were connected to the dual vertex of that triangle. They observed that, just as in our case, the discrete Dirichlet energy of a discrete function coincided with the Dirichlet energy of its linear interpolation on $\Gamma$. However, a crucial difference is that their interpolation was a continuous function, which allowed for a direct comparison with the smooth Dirichlet energy. In our setting, where the interpolations are piecewise linear and non-continuous, we need to employ different techniques. We draw inspiration from nonconforming finite element methods, specifically Braess' analysis on Crouzeix-Raviart elements or nonconforming $P_1$ elements \cite{Br07}. These finite elements correspond to triangulations and provide a discretization of the Poisson equation in a planar domain with boundary, estimating the difference between the discrete and actual solutions in terms of the Dirichlet energy norm. We adapt these ideas to our context, incorporating a careful analysis near the singularities of $\Sigma$, as done by Bobenko and Skopenkov.

We now introduce our version of the Second Lemma of Strang, also known as the Lemma of Berger, Scott, and Strang (Lemma 1.2 in Chapter III of \cite{Br07}). The first term in the inequality corresponds to the \textit{approximation error}, while the second term corresponds to the \textit{consistency error}. Since $U_\Lambda \in M_P$, we can choose $v=U_\Lambda$ to bound the approximation error.

\begin{lemma}\label{lem:strang}
\[\left\|U-u\right\|\leq 2 \inf\limits_{v \in M_P} \left\|U-v\right\|+\sup\limits_{f \in M_0} \frac{\left|a(U,f)\right|}{\left\|f\right\|}\]
\end{lemma}

\begin{proof}
Corollary~\ref{cor:harmonic} implies that $a(u,u-v)=\langle du, df \rangle=0$ with $f:=u-v \in M_0$. Thus, we have
\begin{align*}
\left\|u-v\right\|^2&=a(u-v,u-v)=a(U-v,u-v)+a(u,u-v)-a(U,u-v)\\
&=a(U-v,u-v)-a(U,u-v)\\
\Rightarrow \left\|u-v\right\|&=\frac{a(U-v,u-v)}{\left\|u-v\right\|}-\frac{a(U,u-v)}{\left\|u-v\right\|}\leq \left\|U-v\right\|+\frac{\left|a(U,f)\right|}{\left\|f\right\|}
\end{align*}
by the Cauchy-Schwarz inequality. Applying the triangle inequality, we obtain for any $v \in M_P$:
\begin{align*}
\left\|U-u\right\|&\leq\left\|U-v\right\|+\left\|u-v\right\|\leq 2 \left\|U-v\right\|+\frac{\left|a(U,f)\right|}{\left\|f\right\|}.\qedhere
\end{align*}
\end{proof}

In the forthcoming approximations, we will encounter various constants that depend on parameters such as the smooth function $U$, the regularity parameter $\phi$, or the singularity $O \in S$. To simplify the notation, we denote these constants by $C^{(i)}_{a,b,c}$, where $i$ indexes the constants and $a,b,c$ represent the corresponding parameters. For example, $C^{(1)}_{U,O}$ depends solely on the smooth function $U$ and the singularity $O \in S$ and is the first constant introduced in the subsequent analysis. It is essential to note that all these constants depend on the polyhedral surface $\Sigma$ without any further specification.

To distinguish the constants in the $h$-adapted case, we denote them by $C^{(i,h)}_{a,b,c}$. Although we do not provide explicit expressions for these constants here, it is important to note that we have designed the proofs in a way that allows for easy computation of these constants.

In our subsequent local investigations, we will work with charts to facilitate our analysis. Recall that a discrete chart $z_Q$ of a quadrilateral face $Q \in F(\Lambda)$ is defined as an isometric embedding of $Q$ into the complex plane $\mathbb{C}$. To simplify the notation, we identify the vertices and edges of $Q$ with their corresponding complex values. In cases where $Q$ is not incident to a singularity $O \in S$, we extend the chart $z_Q$ to an isometric mapping $g$ of a neighborhood around $Q$ to $\mathbb{C}$. For the sake of simplicity, we denote the composition $U \circ g^{-1}$ as $U$ in our discussions. In this context, gradients such as $\nabla U$ and $\nabla_Q u$ may depend on the choice of chart, but their lengths $|\nabla U|$, $|\nabla_Q u|$, and $|\nabla U - \nabla_Q u|$ are independent of the specific chart chosen.

Finally, for simplicity, we introduce the notation $|x|_0 := \max\{0, x\}$.


\subsection{The approximation error}\label{sec:approximation}

To bound the approximation error, we follow the approach outlined by Bobenko and Skopenkov in \cite{BoSk12}, with the adaptations for $h$-adapted discrete Riemann surfaces provided by Bobenko and B\"ucking in \cite{BoBu17}.

Let us first recall the projection lemma, which is a minor variant of Lemma 4.2 in \cite{Sk13, BoSk12}:

\begin{lemma}\label{lem:projection}
Let $x_1, x_2 \in \mathbb{C} \cong \mathds{R}^2$ be two linearly independent vectors of length one, and let the angle $\alpha := \min \arccos\left(\pm x_1 \cdot x_2\right)$ be the smaller angle between $x_1$ and $\pm x_2$. Then, for any $y \in \mathds{C}$, we have \[|y| \leq \frac{2}{\sin(\alpha)} \max\limits_{i\in\{1,2\}} |x_i \cdot y|.\]
\end{lemma}

\begin{proof}
Without loss of generality, we can assume that $x_1 \cdot x_2 \geq 0$. Let $\beta_i := \arccos\left(x_i \cdot y/|y|\right)$. Then, $\alpha = \pm \beta_1 \pm \beta_2$ for an appropriate choice of signs.

Since $\alpha \leq \frac{\pi}{2}$, it follows that not both $\beta_1$ and $\beta_2$ can be in the interval $\left(\frac{\pi-\alpha}{2}, \frac{\pi+\alpha}{2}\right)$. Let us assume that $\beta_i$ is not in this interval. Then, we have:
\[|x_i \cdot y|=|y|\cdot|\cos(\beta_i)|\geq |y|\cos\left(\frac{\pi-\alpha}{2}\right)=|y|\sin\left(\frac{\alpha}{2}\right)\geq |y|\frac{\sin(\alpha)}{2}.\qedhere \]
\end{proof}

Next, we introduce the following definition:

\begin{definition}
Let $F: U \subset \mathds{C} \to \mathds{R}$. We define $\left|D^k F(x+iy)\right| := \max_{0 \leq j \leq k} \left|\frac{\partial^k F}{\partial^j x \partial^{k-j} y}(x+iy)\right|$.
\end{definition}

Now, we can state the gradient approximation lemma, which is an analogue of Lemma 4.3 in \cite{BoSk12} or Lemma 4.5 in \cite{Sk13}. The main difference is that we consider quadrilaterals that may be non-convex, and we do not consider their convex hull since it may contain a conical singularity.

\begin{lemma}\label{lem:gradient}
Let $Q\in F(\Lambda)$ be not incident to a conical singularity of $\Sigma$. Then, we have: \[\max_{z \in Q} \left|\nabla U(z)-\nabla_Q U_\Lambda\right|\leq \frac{8\sqrt{2}}{\sin(\phi)} h\max\limits_{z \in Q}\left\|D^2 U(z)\right\|.\]
\end{lemma}

\begin{proof}
Consider the quadrilateral $Q$ with vertices $b_-, w_-, b_+, w_+$ in counterclockwise order. Let $e$ denote one of the two diagonals $b_-b_+$ and $w_-w_+$. We will analyze two cases.

In the case that $e$ is contained in $Q$, we integrate the expression $(\nabla U(z') - \nabla_Q U_\Lambda)\cdot e/|e|$ along $e$ and apply Rolle's theorem. This yields a point $z_0 \in e$ where $(\nabla U(z_0) - \nabla_Q U_\Lambda)\cdot e/|e| = 0$. By using the bound $\left|\nabla U(z') - \nabla U(z_0)\right| \leq 4\sqrt{2}h\max_{z \in Q}\left\|D^2 U(z)\right\|$, we obtain for all $z' \in Q$ the inequality \[\frac{\left|(\nabla U(z')-\nabla_Q U_\Lambda) \cdot e\right|}{|e|}\leq 4\sqrt{2}h\max\limits_{z \in Q}\left\|D^2 U(z)\right\|.\]

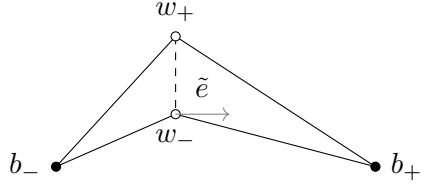
\begin{figure}[htbp]
   \centering
    \beginpgfgraphicnamed{gradient}
			\begin{tikzpicture}[white/.style={circle,draw=black,fill=white,thin,inner sep=0pt,minimum size=1.2mm},
black/.style={circle,draw=black,fill=black,thin,inner sep=0pt,minimum size=1.2mm},
gray/.style={circle,draw=black,fill=gray,thin,inner sep=0pt,minimum size=1.2mm},scale=0.6]
			\clip(-1.7,-4.6) rectangle (7.4,-0.5);
			\draw (-0.6,-4.16)-- (2.02,-1.28);
			\draw (2.02,-1.28)-- (6.4,-4.16);
			\draw (6.4,-4.16)-- (2.02,-3);
			\draw[dashed] (2.02,-1.28)-- (2.02,-3);
			\draw (2.02,-3)-- (-0.6,-4.16);
			\node[white] (w1) [label=below:$w_-$] at (2.02,-3) {};
			\node[white] (w2) [label=above:$w_+$] at (2.02,-1.28) {};
			\node[black] (b1) [label=left:$b_{-}$] at (-0.6,-4.16) {};
			\node[black] (b2) [label=right:$b_{+}$] at (6.4,-4.16) {};
			\draw [color=gray,->] (2.02,-3)-- (3.2,-3) {};
			\draw (2.6,-2.4) node {$\tilde{e}$};
		\end{tikzpicture}
		\endpgfgraphicnamed
   \caption{Application of (generalized) Rolle's theorem if $Q$ is not convex.}
	\label{fig:gradient}
\end{figure}

In the case where $e$ is not contained inside $Q$, we assume that the reflex angle is at $w_-$. We consider the path $b_-w_-b_+$ and integrate $(\nabla U(z') - \nabla_Q U_\Lambda)\cdot (w_--b_-)/|w_--b_-|$ along $b_-w_-$ and $(\nabla U(z') - \nabla_Q U_\Lambda)\cdot (b_+-w_-)/|b_+-w_-|$ along $w_-b_+$. This results in a piecewise linear function that takes the same value at $b_-$ and $b_+$. By applying the generalized Rolle's theorem, we can find a point $z_0$ on $b_-w_-$ or $b_+w_-$ where $(\nabla U(z_0) - \nabla_Q U_\Lambda)\cdot (w_--b_-)/|w_--b_-| = 0$, or \[(\nabla U(w_-)-\nabla_Q U_\Lambda)\cdot \frac{w_--b_-}{|w_--b_-|}\geq 0 \geq (\nabla U(w_-)-\nabla_Q U_\Lambda)\cdot \frac{b_+-w_-}{|b_+-w_-|}\] or the inequality with reversed signs holds true. In either case, there exist $z_0 \in b_-w_-b_+$ and a unit vector $\tilde{e}$ whose argument lies between those of $(w_--b_-)$ and $(b_+-w_-)$, such that $(\nabla U(z_0)-\nabla_Q U_\Lambda)\cdot \tilde{e}=0$. As before, we obtain the inequality\[\left|(\nabla U(z')-\nabla_Q U_\Lambda) \cdot \tilde{e}\right|\leq 4\sqrt{2}h\max\limits_{z \in Q}\left\|D^2 U(z)\right\|\] for all $z' \in Q$. Thus, we have established the inequality \[\left|(\nabla U(z')-\nabla_Q U_\Lambda) \cdot x_i\right|\leq 4\sqrt{2}h\max\limits_{z \in Q}\left\|D^2 u(z)\right\|\] for two unit vectors $x_1,x_2$ (either $e/|e|$ or $\tilde{e}$ for each of the two diagonals). Due to the construction and $\phi$-regularity, the angle between $x_1$ and $x_2$ falls within the range of $\phi$ to $\pi-\phi$. Thus, we can apply Lemma~\ref{lem:projection} to deduce \[\left|\nabla U(z')-\nabla_Q U_\Lambda\right|\leq \frac{8\sqrt{2}}{\sin(\phi)}h\max\limits_{z \in Q}\left\|D^2 u(z)\right\|.\qedhere\]
\end{proof}

\begin{definition}
Let $O \in \Sigma$ be a conical singularity or a regular point of $\Sigma$. We consider the domain $D_O$ of the chart $g_O$ as defined in Section~\ref{sec:Riemann}. In this context, we define the functions $U_O:=U \circ g_O^{-1}$ and $U_p:=U \circ g_p^{-1}$ for any $p \in D_O \backslash {O}$.
\end{definition}

In the subsequent analysis, we fix a point $O \in \Sigma$. Recall that the radius of $D_O$ was denoted by $R_O$. Our goal is to establish results analogous to Lemmas 4.5 to 4.13 presented in the work of Bobenko and Skopenkov \cite{BoSk12}. While we largely follow their proof ideas, we make slight adjustments to accommodate our framework involving quadrilaterals. The lack of boundedness in the partial derivatives of $g_p$ near $O$ prevents us from directly applying Lemma~\ref{lem:gradient} to prove the convergence of the discrete Dirichlet energy.

\begin{lemma}\label{lem:derivative_estimation}
Let $p \in D_O$, $p\neq O$. Then, we have the following estimates:
\begin{align*}
\left\|D^1 U_p(z)|_{z=g_p(p)}\right\|&\leq C^{(1)}_{U,O} r^{\gamma_O-1},\\
\left\|D^2 U_p(z)|_{z=g_p(p)}\right\|&\leq C^{(2)}_{U,O} r^{\gamma_O-2}.
\end{align*}
\end{lemma}

\begin{proof}
Using the chain rule and the fact that $U_p=U\circ g_O^{-1}\circ g_O \circ g_p^{-1}$, we have
\[\left\|D^1 U_p(z)|_{z=g_p(p)}\right\|\leq 2\left\|D^1 U_O(z)|_{z=g_O(p)}\right\|\cdot \left\|D^1  \left(g_O \circ g_p^{-1}\right)(z)|_{z=g_p(p)}\right\|.\]
Applying Leibniz' rule, we obtain for the second derivative
\begin{align*}
\left\|D^2 U_p(z)|_{z=g_p(p)}\right\|&\leq 4\left\|D^2 U_O(z)|_{z=g_O(p)}\right\|\cdot \left\|D^1  \left(g_O \circ g_p^{-1}\right)(z)|_{z=g_p(p)}\right\|^2\\&+2\left\|D^1 U_O(z)|_{z=g_O(p)}\right\|\cdot \left\|D^2  \left(g_O \circ g_p^{-1}\right)(z)|_{z=g_p(p)}\right\|.
\end{align*}

By construction, the map $g_O \circ g_p^{-1}$ is of the form $z\mapsto (az)^{\gamma_O}$ with $|a|=1$. This yields
\begin{align*}
\left\|D^1 \left(g_O \circ g_p^{-1}\right)(z)|_{z=g_p(p)}\right\|&=\gamma_O r^{\gamma_O-1},\\
\left\|D^2 \left(g_O \circ g_p^{-1}\right)(z)|_{z=g_p(p)}\right\|&=\gamma_O(\gamma_O-1) r^{\gamma_O-2}.
\end{align*}

Since $U_O$ extends smoothly to a neighborhood of the disk $g_O(D_O)$ of radius $R_O$ in $\mathds{C}$, we can define $\max_{z\in g_O(\overline{D_O})}|D^k U_O(z)|$. Combining the previous results, we obtain
\begin{align*}
\left\|D^1 U_p(z)|_{z=g_p(p)}\right\|&\leq 2\max_{z\in g_O(\overline{D_O})}\|D^1 U_O(z)\|\cdot \gamma_O r^{\gamma_O-1},\\
\left\|D^2 U_p(z)|_{z=g_p(p)}\right\|&\leq 4\max_{z\in g_O(\overline{D_O})}\|D^2 U_O(z)\|\cdot \gamma_O^2 r^{2\gamma_O-2}+2\max_{z\in g_O(\overline{D_O})}\|D^1 U_O(z)\|\cdot \gamma_O(\gamma_O-1) r^{\gamma_O-2}\\
&\leq 2\gamma_O \left( 2\max_{z\in g_O(\overline{D_O})}\|D^2 U_O(z)\| \gamma_O R_O^{\gamma_O}+\max_{z\in g_O(\overline{D_O})}\|D^1 U_O(z)\|(\gamma_O-1)\right)r^{\gamma_O-2}.\qedhere
\end{align*}
\end{proof}

\begin{lemma}\label{lem:gradient_outside}
Let $Q \subset D_O$ be a quadrilateral that does not intersect the disk of radius $h$ around $O$. If $(\Sigma,\Lambda)$ is $h$-adapted, we change the radius to $h^{1/\gamma_O}$ and assume that $\gamma_O \leq 1/2$. Then, we have the following estimates:
 \begin{align*}
\int_Q \left|\nabla U-\nabla_Q U_\Lambda\right|^2 &\leq C^{(3)}_{U,O,\phi}h\int_Q r^{2\gamma_O-2}drd\psi,\\
\textnormal{and}\quad \int_Q \left|\nabla U-\nabla_Q U_\Lambda\right|^2 &\leq C^{(3,h)}_{u,O,\phi}h\int_Q r^{\gamma_O-1}drd\psi \quad \textnormal{in the $h$-adapted case.}
\end{align*}
\end{lemma}

\begin{proof}
Let $z=(r(z),\psi(z)) \in Q$ denote the vertex of $Q$ closest to $O$. The distance $r$ from any other point in $Q$ to $O$ satisfies $h\leq r(z)\leq r \leq r(z)+2h\leq 3r(z)$. Combining Lemmas~\ref{lem:gradient} and~\ref{lem:derivative_estimation}, we obtain:
\begin{align*}
\int_Q \left|\nabla U-\nabla_Q U_\Lambda\right|^2 &\leq \textnormal{area}(Q)  \max_{w \in Q} \left|\nabla U(w)-\nabla_Q U_\Lambda\right|^2\\
&\leq \frac{128}{\sin^2(\phi)} h^2 \textnormal{area}(Q) \left(C^{(2)}_{U,O}\right)^2 3^{2|\gamma_O-2|_0} r(z)^{2\gamma_O-4}\\
& \leq \frac{128 \left(C^{(2)}_{U,O}\right)^2}{\sin^2(\phi)} h 3^{2|\gamma_O-2|_0}3^{2|2-\gamma_O|_0} \int_Q  h r^{2\gamma_O-4} rdrd\psi\\
& \leq \frac{128 \left(C^{(2)}_{U,O}\right)^2}{\sin^2(\phi)} 3^{2|\gamma_O-2|} h  \int_Q  r^{2\gamma_O-2} drd\psi.
\end{align*}

In the $h$-adapted case, $h^{1/\gamma_O}\leq r(z)\leq r \leq r(z)+2\left(1+\frac{\pi}{2\gamma_O}\right)hr(z)^{1-\gamma_O}\leq \left(3+\frac{\pi}{\gamma_O}\right)r(z)$ by Lemma~\ref{lem:adapted}. Noting that $\gamma_O \leq 1/2$ and $h \leq r^{\gamma_O}$ by assumption, we have
\begin{align*}
\int_Q \left|\nabla U-\nabla_Q U_\Lambda\right|^2 &\leq \frac{128}{\sin^2(\phi)} h^2\left(1+\frac{\pi}{2\gamma_O}\right)^2 r(z)^{2-2\gamma_O} \textnormal{area}(Q) \left(C^{(2)}_{U,O}\right)^2 r(z)^{2\gamma_O-4}\\
& \leq \frac{128 \left(C^{(2)}_{U,O}\right)^2}{\sin^2(\phi)}\left(1+\frac{\pi}{2\gamma_O}\right)^2\left(3+\frac{\pi}{\gamma_O}\right)^{2} h \int_Q  h r^{-2} rdrd\psi\\
& \leq \frac{128 \left(C^{(2)}_{U,O}\right)^2}{\sin^2(\phi)}\left(1+\frac{\pi}{2\gamma_O}\right)^2\left(3+\frac{\pi}{\gamma_O}\right)^{2} h  \int_Q  r^{\gamma_O-1} drd\psi. \qedhere
\end{align*}
\end{proof}

\begin{lemma}\label{lem:contribution_outside}
Let $T'=T'_{O,h}$ and $T'_{(h)}=T'_{O,h,(h)}$ be the sets of all quadrilaterals $Q \subset D_O$ of $\Lambda$ that do not intersect the disks of radius $h$ and $h^{1/\gamma_O}$ (if $\gamma_O<1/2$, otherwise of radius $h$) around $O$, respectively. Additionally, we assume $h \leq R_O$. Then,
\begin{align*}
\sum\limits_{Q \in T'} \int_Q \left|\nabla U-\nabla_Q U_\Lambda\right|^2 \leq C^{(4)}_{U,O,\phi}\lambda_O(h):= C^{(4)}_{U,O,\phi}\cdot\begin{cases}
        h & \text{if }\gamma_O>1/2,\\
        h \left|\log\frac{h}{R_O} \right| & \text{if }\gamma_O=1/2,\\
				h^{2\gamma_O} & \text{if }\gamma_O<1/2.
    \end{cases}
\end{align*}
If $(\Sigma,\Lambda)$ is $h$-adapted and $\gamma_O \leq 1/2$, then
\begin{align*}
\sum\limits_{Q \in T'} \int_Q \left|\nabla U-\nabla_Q U_\Lambda\right|^2 \leq C^{(4,h)}_{U,O,\phi}h.
\end{align*}
\end{lemma}

\begin{proof}
Applying Lemma~\ref{lem:gradient_outside} to each quadrilateral $Q \in T'$ and then extending the integration domain to an annulus, we obtain the following inequalities:
\begin{align*}
\sum\limits_{Q \in T'} \int_Q \left|\nabla U-\nabla_Q U_\Lambda\right|^2 &\leq C^{(3)}_{U,O,\phi}h \sum\limits_{Q \in T'} \int_Q r^{2\gamma_O-2}drd\psi\\
&\leq C^{(3)}_{U,O,\phi}h \int_h^{R_O}\int_0^{2\pi/\gamma_O} r^{2\gamma_O-2}drd\psi=\frac{2\pi C^{(3)}_{U,O,\phi}}{\gamma_O}\int_h^{R_O} r^{2\gamma_O-2}dr\\
&\leq \frac{2\pi C^{(3)}_{U,O,\phi}}{\gamma_O}\cdot \begin{cases}
       h R_O^{2\gamma_O-1}/(2\gamma_O-1) & \text{if }\gamma_O>1/2,\\
       h \left|\log\frac{h}{R_O} \right| & \text{if }\gamma_O=1/2,\\
		 	h^{2\gamma_O}/(1-2\gamma_O) & \text{if }\gamma_O<1/2.
    \end{cases}
\end{align*}
For the $h$-adapted case:
\begin{align*}
\sum\limits_{Q \in T'_{(h)}} \int_Q \left|\nabla U-\nabla_Q U_\Lambda\right|^2 &\leq C^{(3,h)}_{U,O,\phi}h\int_{h^{1/\gamma_O}}^{R_O}\int_0^{2\pi/\gamma_O} r^{\gamma_O-1}drd\psi\leq\frac{2\pi C^{(3,h)}_{U,O,\phi}}{\gamma_O^2}R_O^{\gamma_O}h. \qedhere
\end{align*}
\end{proof}

In the following, we aim to estimate the integral of $\left|\nabla U(z)-\nabla_Q U_\Lambda\right|^2$ close to the singularity $O$. To do this, we first find an upper bound for $\left|\nabla_Q U_{\Lambda}\right|^2$. By the definition of the discrete gradient and using projection Lemma~\ref{lem:projection}, it suffices to bound the difference quotients along the diagonals of a quadrilateral.

\begin{lemma}\label{lem:difference_quotient}
For any two points $x \neq y$ in $D_O$, we have \[\frac{|U(x)-U(y)|}{|xy|}\leq C^{(5)}_{U,O} \max\{|Ox|,|Oy|\}^{\gamma_O-1}.\]
\end{lemma}

\begin{proof}
We use polar coordinates $p=(r,\psi)$ in $D_O$. By Lemma~\ref{lem:derivative_estimation}, we obtain:
\begin{align*}
|U(x)-U(y)|\leq \int_{xy} \sqrt{2}\left\|D^1 U_p(z)|_{z=g_p(p)}\right\|(|dr|+r|d\psi|)\leq \sqrt{2} C^{(1)}_{U,O} \int_{xy} r^{\gamma_O-1} (|dr|+r|d\psi|).
\end{align*}

Without loss of generality, let us assume $|Oy|\geq |Ox|$. If the projection of $O$ onto the line $xy$ lies outside the segment $xy$, then
\[\int_{xy} r^{\gamma_O-1} |dr|\leq \frac{1}{\gamma_O}|Oy|^{\gamma_O}.\] 

If the projection of $O$, denoted by $q$, lies between $x$ and $y$, then
\begin{align*}
\int_{xy} r^{\gamma_O-1} |dr|= \frac{1}{\gamma_O}\left(|Oy|^{\gamma_O}+|Ox|^{\gamma_O}-2|Oq|^{\gamma_O}\right)\leq \frac{2}{\gamma_O}\left(|Oy|^{\gamma_O}-\max\left\{0,|Oy|-|xy|\right\}^{\gamma_O}\right).
\end{align*}

In the case that $|xy|\geq |Oy|$, we obtain
\begin{align*}
|U(x)-U(y)|&\leq \frac{2\sqrt{2}C^{(1)}_{U,O}}{\gamma_O}|Oy|^{\gamma_O} +  \sqrt{2}C^{(1)}_{U,O} \left|\int_{xy} r^{\gamma_O}d\psi\right|\\
&\leq \frac{2\sqrt{2}C^{(1)}_{U,O}}{\gamma_O}|xy||Oy|^{\gamma_O-1} +\sqrt{2}C^{(1)}_{U,O} |xy||Oy|^{\gamma_O-1} |\angle xOy|\\ &\leq \left(\frac{2\sqrt{2}}{\gamma_O}+\sqrt{2}\pi\right) C^{(1)}_{U,O}|xy||Oy|^{\gamma_O-1}.
\end{align*}

If $|xy|<|Oy|$, then $|\angle xOy|$ cannot be the largest angle of the triangle $\triangle xOy$, so $|\angle xOy|< \pi/2$ and $|\angle xOy|<(\pi/2)\sin|\angle xOy|$ by the concavity of the sine function. Applying the sine theorem, we have $|Oy| \sin|\angle xOy|\leq |xy|$. By using Bernoulli's inequality for $\gamma_O<1$ and the monotonicity of the power function for $\gamma_O \geq 1$, we get \[|Oy|^{\gamma_O}-(|Oy|-|xy|)^{\gamma_O}=|Oy|^{\gamma_O}\left(1-(1-|xy|/|Oy|)^{\gamma_O}\right) \leq \max\{1,\gamma_O\}|xy||Oy|^{\gamma_O-1}.\] Hence,
\begin{align*}
|U(x)-U(y)|&\leq \frac{2\sqrt{2}C^{(1)}_{U,O}}{\gamma_O}\max\{1,\gamma_O\}|xy||Oy|^{\gamma_O-1} +  \sqrt{2}C^{(1)}_{U,O}|Oy|^{\gamma_O} |\angle xOy|\\
&\leq \frac{2\sqrt{2}C^{(1)}_{U,O}}{\gamma_O}\max\{1,\gamma_O\}|xy||Oy|^{\gamma_O-1} +  \frac{\pi\sqrt{2}C^{(1)}_{U,O}}{2}|Oy|^{\gamma_O-1} |xy|. \qedhere
\end{align*}
\end{proof}

We now proceed with geometric bounds for a quadrilateral $Q$ in a $\phi$-regular discretization of $\Sigma$.

\begin{lemma}\label{lem:geometric_bounds}
Let $Q \subset D_O$ be a quadrilateral with vertices $w,x,y,z$, and let $e$ be any of its four edges. Let $z$ be the vertex of maximum distance to the singularity $O$. Then,
\[ |e|^2 \leq \frac{2}{\sin^3(\phi)} \textnormal{area}(Q) \quad \textnormal{and} \quad |Oz|\leq \frac{2+\sin(\phi)}{\sin(\phi)}\max\{|Ow|,|Oy|\}.\]
\end{lemma}
\begin{proof}
If $d_1$ and $d_2$ are the two diagonals of $Q$, and $\varphi_Q$ is their intersection angle, consider the triangle $\triangle$ formed by $e$ and one of the diagonals, say $d_i$. While $\triangle$ may lie outside $Q$, its interior angle opposite to $d_i$ is greater than or equal to an interior angle of $Q$ opposite to $d_i$. If this angle is greater than $\pi/2$, then $|d_i|\geq |e|$. Otherwise, we can apply the sine theorem and use $\phi$-regularity to deduce that:
\[\frac{|e|}{|d_i|}\leq \frac{1}{\sin(\phi)} \Rightarrow |e|^2 \leq \frac{1}{\sin^3(\phi)} |d_1| |d_2|\sin(\varphi_Q) = \frac{2}{\sin^3(\phi)}\text{area}(Q).\]

Let us assume that $|Oy|\geq |Ow|$. Using the above result for the diagonal $wy$ and triangle inequality, \[|Oz|-|Oy|\leq |zy| \leq \frac{|wy|}{\sin(\phi)}\leq \frac{|Ow|+|Oy|}{\sin(\phi)}\leq \frac{2}{\sin(\phi)}|Oy|. \qedhere\]
\end{proof}

\begin{lemma}\label{lem:gradient_inside}
Let $Q \subset D_O$ be a quadrilateral with vertices $w,x,y,z$. Then, \[\int_Q \left|\nabla_Q U_\Lambda\right|^2 \leq C^{(6)}_{U,O,\phi}\int_Q r^{2\gamma_O-2}rdrd\psi.\]
\end{lemma}

\begin{proof}
Without loss of generality, we assume $|Oy|\geq |Ow|$, and that $z$ has the maximum distance to $O$. Lemma~\ref{lem:difference_quotient} provides bounds for the two discrete difference quotients of $U_\Lambda$ along the diagonals. By applying Lemma~\ref{lem:projection} together with Lemma~\ref{lem:geometric_bounds}, we obtain the following bound for the discrete gradient: \[\left|\nabla_Q U_\Lambda\right| \leq \frac{2C^{(5)}_{U,O}}{\sin(\phi)}\max\left\{|Oz|^{\gamma_O-1},|Oy|^{\gamma_O-1}\right\}\leq \frac{2C^{(5)}_{U,O}}{\sin(\phi)}\left(\frac{2+\sin(\phi)}{\sin(\phi)}\right)^{|1-\gamma_O|_0} |Oz|^{\gamma_O-1}.\]

If $\gamma_O \leq 1$, $|Oz|^{\gamma_O-1} \leq r^{\gamma_0-1}$ for all $p=(r,\psi) \in Q$. If $\gamma_O > 1$, we consider the homothety $Z$ of $Q$ with center $z$ and ratio $1/4$. Then, $|Oz| \leq 2r$ for all $p=(r,\psi) \in Z(Q)$, so \[|Oz|^{2\gamma_O-2}\text{area}(Q)=16|Oz|^{2\gamma_O-2}\text{area}(Z(Q))\leq 16 \int_{Z(Q)} (2r)^{2\gamma_O-2} rdrd\psi \leq 2^{2\gamma_O+2}\int_Q r^{2\gamma_O-2}rdrd\psi.\]

In particular,
\begin{align*}
\int_Q \left|\nabla_Q U_\Lambda\right|^2 &\leq \frac{4\left(C^{(5)}_{U,O}\right)^2}{\sin^2(\phi)} \left(\frac{2+\sin(\phi)}{\sin(\phi)}\right)^{2|1-\gamma_O|_0} |Oz|^{2\gamma_O-2} \text{area}(Q)\\
&\leq \frac{4\left(C^{(5)}_{U,O}\right)^2}{\sin^2(\phi)} \left(\frac{2+\sin(\phi)}{\sin(\phi)}\right)^{2|1-\gamma_O|_0} 2^{4+2|\gamma_O-1|_0} \int_Q r^{2\gamma_O-2}rdrd\psi. \qedhere
\end{align*}
\end{proof}

\begin{lemma}\label{lem:contribution_inside}
Let $T=T_{O,h}$ and $T_{(h)}=T_{O,h,(h)}$ be the sets of all quadrilaterals $Q \subset D_O$ of $\Lambda$ that intersect the disks of radius $h$ and $h^{1/\gamma_O}$ (if $\gamma_O<1/2$, otherwise of radius $h$) around $O$, respectively. Then, we have the following estimates:
\begin{align*}
\sum_{Q \in T} \int_Q \left|\nabla U-\nabla_Q U_\Lambda\right|^2 &\leq C^{(7)}_{U,O,\phi} h^{2\gamma_O}\\
\textnormal{and} \sum_{Q \in T_{(h)}} \int_Q \left|\nabla U-\nabla_Q U_\Lambda\right|^2 &\leq C^{(7,h)}_{U,O,\phi} h^{2} \quad \quad \textnormal{if $(\Sigma,\Lambda)$ is $h$-adapted and $\gamma_O \leq 1/2$.}
\end{align*}
\end{lemma}

\begin{proof}
We can see that for any $Q \in T$, it is contained within a disk of radius $3h$ around $O$. Hence, using Lemmas~\ref{lem:derivative_estimation} and~\ref{lem:gradient_inside}, we obtain the following bound:
\begin{align*}
\sum_{Q \in T} \int_Q \left|\nabla U-\nabla_Q U_\Lambda\right|^2 & \leq \sum_{Q \in T} \int_Q \left(\left|\nabla U\right|^2+\left|\nabla_Q U_\Lambda\right|^2\right) 
\leq \left(2 \left(C^{(1)}_{U,O}\right)^2+C^{(6)}_{U,O,\phi}\right)\int_0^{3h} \int_0^{2\pi/\gamma_O} r^{2\gamma_O-2}rdrd\psi\\
&=  3^{2\gamma_O}\pi\frac{2\left(C^{(1)}_{U,O}\right)^2+C^{(6)}_{U,O,\phi}}{\gamma_O^2} h^{2\gamma_O}.
\end{align*}

In the case of $h$-adaptedness, Lemma~\ref{lem:adapted} shows that any $Q \in T_{(h)}$ is contained within the disk of radius $\left(3+\frac{\pi}{\gamma_O}\right)h^{1/\gamma_O}$ around $O$. Therefore, we obtain a similar bound:
\begin{align*}
\sum_{Q \in T_{(h)}} \int_Q \left|\nabla U-\nabla_Q U_\Lambda\right|^2 & \leq \left(2 \left(C^{(1)}_{U,O}\right)^2+C^{(6)}_{U,O,\phi}\right)\int_0^{\left(3+\frac{\pi}{\gamma_O}\right)h^{1/\gamma_O}} \int_0^{2\pi/\gamma_O} r^{2\gamma_O-2}rdrd\psi\\
&=  \left(3+\frac{\pi}{\gamma_O}\right)^{2\gamma_O}\pi\frac{2 \left(C^{(1)}_{U,O}\right)^2+C^{(6)}_{U,O,\phi}}{\gamma_O^2} h^{2}.\qedhere
\end{align*}
\end{proof}

As observed in Lemma~\ref{lem:contribution_outside}, the convergence rate depends on the conical singularity of $\Sigma$ with the smallest singularity index.

\begin{definition}
Let $\gamma_{\Sigma}:=\min\{1,\min_{O \in S} \gamma_O\}$ denote the smallest and $\tilde{\gamma}_\Sigma:=\max\{1,\max_{O \in S} \gamma_O\}$ the largest singularity index among the conical singularities $S$ of $\Sigma$. If there is no conical singularity, then $\Sigma$ is a torus. For $h>0$, we denote
\[
\lambda_\Sigma(h):=
\begin{cases}
h & \text{ if } \gamma_\Sigma>1/2,\\
2h\left| \log h \right| & \text{ if } \gamma_\Sigma=1/2,\\
h^{2\gamma_\Sigma} & \text{ if } \gamma_\Sigma<1/2.
\end{cases}
\]
\end{definition}

Now, combining Lemmas~\ref{lem:contribution_outside} and \ref{lem:contribution_inside}, we can estimate the approximation error in the Second Lemma of Strang, Lemma~\ref{lem:strang}.

\begin{proposition}\label{prop:approximation_error}
If the maximum edge length $h$ of the $\phi$-regular discretization $(\Sigma,\Lambda)$ satisfies $h \leq C^{(8)}$ or $h \leq C^{(8,h)}$ in the $h$-adapted case, then
\begin{align*}
\left\|U-U_\Lambda\right\|^2&\leq C^{(9)}_{U,\phi}\lambda_\Sigma(h)\\
\textnormal{and}\quad \left\|U-U_\Lambda\right\|^2&\leq C^{(9,h)}_{U,\phi}h \quad \textnormal{if $(\Sigma,\Lambda)$ is $h$-adapted}.
\end{align*}
\end{proposition}

\begin{proof}
We start by considering a finite cover of $\Sigma$ with open disks $D_O$ of radius $R_O$, $O \in \Sigma$, containing no singularities other than possibly $O$. Let $S'$ denote the set of all $O$ for which $D_O$ is in the above open cover and assume that $S'$ is minimal, i.e., there are no superfluous disks in the cover. Let $\delta$ denote the Lebesgue number of the cover and let \[c:=\min\left\{\frac{\delta}{2},\left\{\frac{1}{R_O}\right\}_{O \in S'}\right\}.\]

Next, we define appropriate constants for the different cases of $\gamma_\Sigma$:
\begin{itemize}
\item If $\gamma_\Sigma>1/2$, let $C^{(8)}:=\min\{1,c\}$.
\item If $\gamma_\Sigma=1/2$, let $C^{(8)}:=\min\{e^{-1/2},c\}$, where $e$ is Euler's number.
\item If $\gamma_\Sigma<1/2$, we define $K(\Sigma)=1$ if no singularity $O$ has index $\gamma_O=1/2$, and otherwise, we choose $K(\Sigma)>0$ such that $2 \log(t) \leq t^{1-2\gamma_\Sigma}$ for all $t\geq K(\Sigma)$. Then, $C^{(8)}:=\min\{K(\Sigma)^{-1},c\}$.
\end{itemize}

Clearly, $h \leq R_O$ for all $O \in S'$. For any quadrilateral $Q \in F(\Lambda)$, there is $O \in S'$ such that $Q \subset D_O$ by Lebesgue's number lemma. So Lemmas~\ref{lem:contribution_outside} and~\ref{lem:contribution_inside} yield the desired estimate:
\begin{align*}
\left\|U-U_\Lambda\right\|^2&=\sum_{Q \in F(\Lambda)} \int_Q \left|\nabla U-\nabla_Q U_\Lambda\right|^2\\
&\leq \sum_{O \in S'} \left(\sum_{Q \in T'_{O,h}} \int_Q \left|\nabla U-\nabla_Q U_\Lambda\right|^2 + \sum_{Q \in T_{O,h}} \int_Q \left|\nabla U-\nabla_Q U_\Lambda\right|^2\right)\\
& \leq \sum_{O \in S'} \left( C^{(4)}_{U,O,\phi} \lambda_O(h)+ C^{(7)}_{U,O,\phi} h^{2\gamma_O}\right)\\
& \leq \sum_{O \in S'} \left( C^{(4)}_{U,O,\phi} + C^{(7)}_{U,O,\phi}\right)\cdot \begin{cases}
h & \text{ if } \gamma_\Sigma>1/2,\\
2 h\left| \log h \right| & \text{ if } \gamma_\Sigma=1/2,\\
h^{2\gamma_\Sigma} & \text{ if } \gamma_\Sigma<1/2.
\end{cases}
\end{align*}
To arrive at the last line, we used the following: Since $0 < h \leq 1$, $h^{x}$ is a decreasing function in $x$. If $h \leq e^{-1/2}$, then $h^x \leq h \leq 2h|\log h|$ for all $x \geq 1$. Due to $h\leq R_O^{\pm 1}$, $h|\log hR_O^{-1}| \leq h|\log h|+h|\log R_O|\leq 2 h|\log h|$. By construction, if $\gamma_\Sigma<1/2$ and $\gamma_O =1/2$, $\lambda_O(h) \leq 2 h|\log h| \leq h^{2\gamma_\Sigma}$ if $h \leq C^{(8)}$. 

In the $h$-adapted case, the proof follows similarly, with $C^{(8,h)}:=\min\{1,c\}$.
\end{proof}

\begin{remark}
$\phi$-regularity around conical singularities is essential for the convergence of the Dirichlet energies in Proposition~\ref{prop:approximation_error} if there is a conical singularity with $\gamma_O<1/2$. We illustrate this with an example adapted from Bobenko and Skopenkov's work on triangular decompositions \cite{BoSk12} to our setting, where we need to bound the interior angles of quadrilaterals as well as the intersection angles of diagonals.

Let $\gamma_O<1/2$ and $k>1/(1-2\gamma_O)$. Consider the harmonic function $U(r,\psi):=r^{\gamma_O}\cos(\gamma_O\psi)$ in $D_O$. It has bounded Dirichlet energy in this domain.

First, let $Q$ be the rhombus with black vertices $O$ and $(h^k,0)$ such that the white diagonal has length $h$. The edges of $Q$ have length $\sqrt{h^2+h^{2k}}/2<\sqrt{2}h/2<h$ if $h<1$. $U$ takes identical values on the two white vertices. Thus, $\nabla_Q U_{\Lambda}=h^{k(\gamma_0-1)}$. The area of $Q$ is $h^{k+1}/2$. Hence, $2\int_Q \left|\nabla_Q U_{\Lambda}\right|^2=h^{2k(\gamma_0-1)+k+1}$. Since $2k(\gamma_0-1)+k+1=1-k(1-2\gamma_O)<0$, the discrete Dirichlet energy of $U_{\Lambda}$ diverges to infinity as $h\to 0$. Note that the interior angle of $Q$ at $O$ goes to zero as $h \to 0$.

Second, let $Q$ be the rectangle with one black vertex $b_-$ at $O$ and its two white vertices $w_-$ and $w_+$ being at $(h^k,0)$ and $(h,\pi/2)$, respectively. The edges of $Q$ have length $h$ and $h^k<h$ if $h<1$. As $h \to 0$, the intersection angle of the black and the white diagonal approaches zero. We can bound the length of $\nabla_Q U_\Lambda$ from below by considering the discrete difference quotient of $U_\Lambda$ between $(b_-+w_+)/2$ and $(b_++w_-)/2$. This leads to the inequality
\begin{align*}
|\nabla_Q U_\Lambda| &\geq \frac{\left(\sqrt{h^2+h^{2k}}\right)^{\gamma_O}\cos(\gamma_O\angle w_-b_-b_+)+h^{k\gamma_O}-h^{\gamma_O}\cos(\gamma_O\pi/2)}{2h^k}\\
&> \frac{h^{\gamma_O-k}}{2}\left(\left(1+h^{2(k-1)}\right)^{\gamma_O/2} \cos(\pi/4)+h^{(k-1)\gamma_O} -\cos(\pi/4)\right)\\
&> \frac{h^{\gamma_O-k}}{2\sqrt{2}}\left(\left(1+h^{2(k-1)}\right)^{\gamma_O/2} -1 +h^{(k-1)\gamma_O}\right) > \frac{h^{k(\gamma_O-1)}}{2\sqrt{2}}.
\end{align*}
The area of the rectangle is $h^{1+k}$, which leads to $8\int_Q |\nabla_Q U_\Lambda|^2 > h^{2k(\gamma_O-1)+k+1} \to \infty$ as $h \to 0$.
\end{remark}


\subsection{The consistency error}\label{sec:consistency}

Now let us bound the consistency error in Lemma~\ref{lem:strang}. We want to estimate $\left|a(U,f)\right|$ for an arbitrary function $f:V(\Lambda)\to\mathds{R}$.

The function $f$ induces a corresponding function $f:V(X)\to\mathds{R}$ on the vertices of the medial graph by taking the arithmetic mean: If $e=bw$ is an edge of $\Lambda$ and we identify $e$ with its midpoint, then $f(e):=(f(b)+f(w))/2$. On the Varignon parallelogram $F_Q \subset \Sigma$ associated with a quadrilateral $Q \in F(\Lambda)$, we can linearly interpolate $f$. We extend this interpolation to the entire $Q \subset \Sigma$. Doing this for all quadrilaterals $Q \in F(\Lambda)$ gives us a well-defined function in $L^2(\Sigma)$, which we also denote as $f$. By construction, we can choose a representative of $f$ that is continuous at midpoints of edges (i.e., at vertices of $X$), where it agrees with the original function $f:V(X)\to\mathds{R}$. Similarly, $U_\Lambda \in L^2(\Sigma)$ is defined as a piecewise linear function that is continuous along midpoints of edges. On a given quadrilateral $Q$, we always consider $f$ and $U_\Lambda$ to be continuous also at the boundary. This means, in particular, that we consider generally different representatives of $f,U_\Lambda \in L^2(\Sigma)$ on the single quadrilaterals.

Functions $f \in L^2(\Sigma)$ are analogous to the \textit{Crouzeix-Raviart elements} used for quadrilaterals. The Crouzeix-Raviart elements are introduced by Braess in \cite{Br07} as the simplest nonconforming elements for discretizing the Poisson problem. The term \textit{nonconforming} indicates that the discrete functions are not continuous, unlike the \textit{conforming} elements, which are continuous and piecewise linear. The conforming elements correspond to the linear interpolations of $f$ and $u$ on triangles, as considered by Bobenko and Skopenkov in \cite{BoSk12}.

To bound $\left|a(U,f)\right|$, we adopt the ideas from Braess \cite{Br07} for the Crouzeix-Raviart elements on a planar triangulation and adapt them to our setting. Let $T_{O,h}$ be the set of all quadrilaterals $Q \subset D_O$ of $\Lambda$ that intersect the disk of radius $h$ around the conical singularity $O$. We define $\tilde{T}$ as the complement of $\bigcup_{O \in S} T_{O,h}$ in $F(Q)$.

Though $\nabla U$ and $\nabla_Q f$ depend on the chosen chart, their scalar product and lengths are chart-independent, as any two (discrete) charts differ by an isometry. Therefore, we can work with these expressions without specifying the chart.

To estimate the consistency error in Lemma~\ref{lem:strang}, we use Green's first identity, denoting by $n$ the outer unit normal vector on the edges of the quadrilateral $Q$:
\[
\int_Q \nabla U \cdot \nabla_Q f=\int_{\partial Q} (\partial_n U) f-\int_Q \triangle u f=\sum_{e \in \partial Q} \int_e f\partial_n u
\]
since $U$ is harmonic. The normal derivative $\partial_n U$ is chart-independent and well-defined on the edges of a quadrilateral $Q$. However, it changes sign on the same edge when considering the other quadrilateral $Q'$ sharing $e$. Thus, the sum of the expressions $\int_e f(e)\partial_n U$ corresponding to the two quadrilaterals sharing the edge $e$ cancels out. Hence,
\begin{align}
a(U,f)&=\sum\limits_{Q \in F(\Lambda)} \int_Q \nabla U \cdot \nabla_Q f \notag \\
&=\sum\limits_{O \in S}\sum\limits_{Q \in T_{O,h}} \int_Q \nabla U \cdot \nabla_Q f - \sum\limits_{O \in S}\sum\limits_{Q \in T_{O,h}} \sum_{e \in \partial Q} \int_e f(e)\partial_n U + \sum\limits_{Q \in \tilde{T}} \sum_{e \in \partial Q}\int_e (f-f(e))\partial_n U . \tag{$\star$} \label{eq:consistency}
\end{align}

Now we focus on the last term. Since $f$ is linear on $Q$, it is also linear on the edge $e \subset Q$, and thus $\int_e (f-f(e))=0$ because it is zero at the midpoint of $e$. We can subtract a constant from $\partial_n U$ for each oriented edge without changing the sum's value. We choose to subtract $\partial_n U_\Lambda$.
\begin{align*}
\left|\sum\limits_{Q \in \tilde{T}} \sum_{e \in \partial Q} (f-f(e))\partial_n U\right|&\leq\sum\limits_{Q \in \tilde{T}} \sum_{e \in \partial Q} \int_{e} \left|\left(f-f(e)\right)\partial_n (U-U_\Lambda)\right|\\
&\leq \sum\limits_{Q \in \tilde{T}} \sum_{e \in \partial Q} \sqrt{\int_{e} \left(f-f(e)\right)^2} \sqrt{\int_{e} \left(\partial_n (U-U_\Lambda)\right)^2}\\
&\leq \sum\limits_{Q \in \tilde{T}} \sum_{e \in \partial Q} \sqrt{2\int_{0}^{|e|/2} \left|\nabla_Q f \cdot \frac{e}{|e|}\right|^2 s^2ds} \sqrt{\int_{e} \left|\nabla (U-U_\Lambda)\right|^2}\\
&\leq \sum\limits_{Q \in \tilde{T}} \sum_{e \in \partial Q} \sqrt{\frac{\left|\nabla_Q f\right|^2 |e|^2}{12}} \sqrt{|e|\int_{e} \left|\nabla (U-U_\Lambda)\right|^2}\\
&\leq \sqrt{\sum\limits_{Q \in \tilde{T}} \sum_{e \in \partial Q} \frac{\left|\nabla_Q f\right|^2 |e|^2}{12}} \sqrt{\sum\limits_{Q \in \tilde{T}} \sum_{e \in \partial Q}|e|\int_{e} \left|\nabla (U-U_\Lambda)\right|^2}\\
&\leq \sqrt{\sum\limits_{Q \in \tilde{T}} \frac{2 \left|\nabla_Q f\right|^2 \text{area}(Q)}{3\sin^3{\phi}}} \sqrt{\sum\limits_{Q \in \tilde{T}} \sum_{e \in \partial Q}|e|\int_{e}\left| \nabla (U-U_\Lambda)\right|^2}\\
&= \sqrt{\frac{2}{3\sin^3{\phi}}}\left\|f\right\|_{\tilde{T}} \sqrt{\sum_{Q \in \tilde{T}} \sum_{e \in \partial Q}|e|\int_{e}\left| \nabla (U-U_\Lambda)\right|^2}
\end{align*}
by Cauchy-Schwarz inequality and the geometric bound in Lemma~\ref{lem:geometric_bounds}. Here, $|e|$ is the length of edge $e$, and $\left\|f\right\|_{\tilde{T}}^2:=\sum_{Q \in \tilde{T}} \left|\nabla_Q f\right|^2 \text{area}(Q)$. In the case of $h$-adapted discrete Riemann surfaces, replace the sets $T_{O,h}$ by the sets $T_{O,h,(h)}$ corresponding to the radius $h^{1/\gamma_O}$ if $\gamma_O < 1/2$.

In the next step, we proceed to bound the consistency error, following a similar approach as in the previous Section~\ref{sec:approximation}. In the following, we focus on a conical singularity $O \in \Sigma$ and the corresponding chart $(D_O,g_O)$, where $D_O$ is a disk with radius $R_O$. We start by establishing the analogue of Lemmas~\ref{lem:gradient_outside}.

\begin{lemma}\label{lem:gradient_outside2}
For a quadrilateral $Q \subset D_O$ that does not intersect the disk of radius $h$ around $O$, we have the following bound: \[\sum_{e \in \partial Q}|e|\int_{e} \left|\nabla U-\nabla_Q U_\Lambda\right|^2 \leq \frac{8 C^{(3)}_{U,O,\phi}}{\sin^3{\phi}}h\int_Q r^{2\gamma_O-2}drd\psi.\]
\end{lemma}

\begin{proof}
We make use of the geometric bound from Lemma~\ref{lem:geometric_bounds}. Specifically, we have: \[\sum_{e \in \partial Q}|e|\int_{e}\left|\nabla U-\nabla_Q U_\Lambda\right|^2\leq \sum_{e \in \partial Q}|e|^2\max_{w \in Q}\left| \nabla U(w)-\nabla_Q U_\Lambda\right|^2\leq \frac{8}{\sin^3{\phi}}\text{area}(Q)\max_{w \in Q} \left| \nabla U(w)-\nabla_Q U_\Lambda\right|^2.\] Up to the constant factor, this expression is precisely the same one we bounded in Lemma~\ref{lem:gradient_outside}.
\end{proof}

Using Lemma~\ref{lem:gradient_outside2}, we prove the following lemma in exactly the same way as Lemma~\ref{lem:contribution_outside}.

\begin{lemma}\label{lem:contribution_outside2}
Let $T'=T'_{O,h}$ and $T'_{(h)}=T'_{O,h,(h)}$ be the sets of all quadrilaterals $Q \subset D_O$ of $\Lambda$ that do not intersect the disks of radius $h$ and $h^{1/\gamma_O}$ (if $\gamma_O<1/2$, otherwise of radius $h$) around $O$, respectively. Additionally, we assume $h \leq R_O$. Then,
\begin{align*}
\sum\limits_{Q \in T'} \sum_{e \in \partial Q}|e|\int_{e}\left|\nabla u-\nabla_Q u_\Lambda\right|^2 &\leq \frac{8  C^{(4)}_{U,O,\phi}}{\sin^3{\phi}}\lambda_O(h)\\
\textnormal{and} \quad  \sum\limits_{Q \in T'_{(h)}} \sum_{e \in \partial Q}|e|\int_{e}\left|\nabla u-\nabla_Q u_\Lambda\right|^2 &\leq \frac{8  C^{(4,h)}_{U,O,\phi}}{\sin^3{\phi}}h \quad \textnormal{if $(\Sigma,\Lambda)$ is $h$-adapted and $\gamma_O \leq 1/2$.}
\end{align*}
\end{lemma}

We can then derive the following corollary, which is shown in a similar way as Proposition~\ref{prop:approximation_error}.

\begin{corollary}\label{cor:contribution_outside3}
\begin{align*}
\sum_{Q \in \tilde{T}} \sum_{e \in \partial Q}|e|\int_{e}\left| \nabla (U-U_\Lambda)\right|^2 &\leq \frac{8}{\sin^3{\phi}}\lambda_{\Sigma}(h)\left(\sum\limits_{O \in S}C^{(4)}_{U,O,\phi}\right) \quad \textnormal{if $h \leq C^{(8)}$}\\
\textnormal{and} \quad \sum_{Q \in \tilde{T}} \sum_{e \in \partial Q}|e|\int_{e}\left| \nabla (U-U_\Lambda)\right|^2 &\leq \frac{8}{\sin^3{\phi}}h\left(\sum\limits_{O \in S}C^{(4,h)}_{U,O,\phi}\right) \quad \textnormal{if $(\Sigma,\Lambda)$ is $h$-adapted and $h \leq C^{(8,h)}.$}
\end{align*}
Here, we define $C^{(4,h)}_{U,O,\phi}:=C^{(4)}_{U,O,\phi}$ if $\gamma_O \geq 1/2$.
\end{corollary}

Now, let us focus on the contribution of $Q \in T_{O,h}$, where $O \in S$. We define $\left\|f\right\|_{T_{O,h}}^2:=\sum_{Q \in T_{O,h}} \int_Q \left|\nabla_Q f\right|^2$ and $\left\|f\right\|_{S}^2:=\sum_{O \in S} \left\|f\right\|_{T_{O,h}}^2$.

\begin{lemma}\label{lem:contribution_inside2}
\begin{align*}
\left|\sum\limits_{O \in S}\sum\limits_{Q \in T_{O,h}} \int_Q \nabla U \cdot \nabla_Q f\right|&\leq C^{(10)}_{U,\phi}\sqrt{\lambda_\Sigma(h)}\|f\|_{S} \quad \textnormal{if $h \leq C^{(8)}$}\\
\textnormal{and} \quad \left|\sum\limits_{O \in S}\sum\limits_{Q \in T_{O,h,(h)}} \int_Q \nabla U \cdot \nabla_Q f\right|&\leq  C^{(10,h)}_{U,\phi} h\|f\|_{S}\quad \textnormal {if $(\Sigma,\Lambda)$ is $h$-adapted and $h \leq C^{(8,h)}$.}
\end{align*}
\end{lemma}

\begin{proof}
First, observe that any $Q \in T_{O,h}$ is contained in the disk of radius $3h$ around $O$. Using Cauchy-Schwarz inequality and Lemma \ref{lem:derivative_estimation}, we have:
\begin{align*}
\left|\sum\limits_{Q \in T_{O,h}} \int_Q \nabla U \cdot \nabla_Q f\right| &\leq \sum\limits_{Q \in T_{O,h}} \sqrt{\int_Q \left|\nabla U\right|^2} \sqrt{\int_Q \left|\nabla_Q f\right|^2}\\
&\leq \sqrt{\sum\limits_{Q \in T_{O,h}} \int_Q \left|\nabla U\right|^2} \sqrt{\sum\limits_{Q \in T_{O,h}} \int_Q \left|\nabla_Q f\right|^2}\\
&\leq \|f\|_{ T_{O,h}}\sqrt{2} C^{(1)}_{U,O}\sqrt{\int_0^{3h} \int_0^{2\pi/\gamma_O} r^{2\gamma_O-2}rdrd\psi}
&= 3^{\gamma_O}\sqrt{2\pi}\frac{C^{(1)}_{U,O}}{\gamma_O} h^{\gamma_O}\|f\|_{ T_{O,h}}.
\end{align*}
The results now follow from the inequality of the arithmetic and the quadratic mean for the $\|f\|_{T_{O,h}}$ and $h \leq C^{(8)}$ in a similar way as in the proof of Proposition~\ref{prop:approximation_error}. The $h$-adapted case follows similarly.
\end{proof}

\begin{lemma}\label{lem:contribution_inside3}
\begin{align*}
\left|\sum\limits_{O \in S}\sum\limits_{Q \in T_{O,h}}\sum\limits_{e \in \partial Q}\int_e f(e)\partial_n U \right|&\leq C^{(11)}_{U,\phi}\sqrt{\lambda_\Sigma(h)}\|f\|_{S} \quad \textnormal{if $h \leq C^{(8)}$}\\
\textnormal{and} \quad \left|\sum\limits_{O \in S}\sum\limits_{Q \in T_{O,h,(h)}}\sum\limits_{e \in \partial Q}\int_e f(e)\partial_n U \right|&\leq  C^{(11,h)}_{U,\phi} h\|f\|_{S}\quad \textnormal {if $(\Sigma,\Lambda)$ is $h$-adapted and $h \leq C^{(8,h)}$.}
\end{align*}
\end{lemma}

\begin{proof}
Let $Q \subset D_O$ be a quadrilateral with vertices $w,x,y,z$. Since $U$ is harmonic, $\sum_{e \in \partial Q} \int_e \partial_n U=0$ follows from Green's first identity. In particular, we can add a constant to all $f(e)$ without changing $\sum_{e \in \partial Q} \int_e f(e) \partial_n U=0$. We choose this constant to be $(f(w)+f(x)+f(y)+f(z))/4$. Then, the four values of $f$ on the four edges become $\pm((f(z)-f(x))\pm(f(y)-f(w)))/4$. Their absolute value can be bounded from above by: \[\left|\nabla_Q f\right|\frac{|d_1|+|d_2|}{4}\leq \left|\nabla_Q f\right|\frac{\sqrt{|d_1||d_2|}}{2}\leq \left|\nabla_Q f\right|\frac{\sqrt{\text{area}(Q)}}{\sqrt{2 \sin(\phi)}}\] using $\phi$-regularity of $Q$. Here, $d_1, d_2$ denote the two diagonals of $Q$.

Without loss of generality, assume $z$ is the vertex of maximum distance to $O$. Since $\left|\partial_n U\right|\leq |\nabla U|$, we can use the same constant $C^{(5)}{U,O}$ and the same ideas as in Lemma~\ref{lem:difference_quotient} to bound $\sum{e \in \partial Q} \int_e \left|\partial_n U\right|$ by:
\[C^{(5)}_{U,O} \left(|yz||Oz|^{\gamma_O-1}+|zw||Oz|^{\gamma_O-1}+|wx|\max\{|Ow|,|Ox|\}^{\gamma_O-1}+|xy|\max\{|Ox|,|Oy|\}^{\gamma_O-1}\right).\]

Let $p,p'$ be two adjacent vertices among $w,x,y$. Assume $|Op|\geq|Op'|$. By triangle inequality,
\[|Oz|-|Op|\leq|pz|=\frac{|pz|}{|pp'|}|pp'|\leq \frac{|pz|}{|pp'|}(|Op|+|Op'|)\leq 2\frac{|pz|}{|pp'|}|Op|\Rightarrow |Oz| \leq \left(1+ 2\frac{|pz|}{|pp'|}\right)|Op|.\]
Together with the bound of edge lengths in Lemma~\ref{lem:geometric_bounds} (which also bounds half the length of a diagonal), we now deduce:
\[|pp'||Op|^{\gamma_O-1}\leq |pp'|^{\gamma_O}|pp'|^{1-\gamma_O}\left(1+ 2\frac{|pz|}{|pp'|}\right)^{|1-\gamma_O|_0}\leq \sqrt{\frac{2}{\sin^3(\phi)}\text{area}(Q)}5^{|1-\gamma_O|_0}|Oz|^{\gamma_O-1}.\]

We thus obtain: \[\sum\limits_{e \in \partial Q}\int_e |\partial_n U| \leq 2C^{(5)}_{U,O} \left(1+5^{|1-\gamma_O|_0}\right)\sqrt{\frac{2}{\sin^3(\phi)}\text{area}(Q)}|Oz|^{\gamma_O-1}.\]

In the following, we denote by $z_Q$ the vertex of the quadrilateral $Q$ that has the maximum distance to $O$. Combining the above two results and applying the Cauchy-Schwarz inequality again yields:
\begin{align*}
\left|\sum\limits_{Q \in T_{O,h}}\sum\limits_{e \in \partial Q}\int_e f(e)\partial_n U \right| &\leq \sum\limits_{Q \in T_{O,h}} \frac{2C^{(5)}_{U,O} \left(1+5^{|1-\gamma_O|_0}\right)}{\sin^2(\phi)}|\nabla_Q f||Oz_Q|^{\gamma_O-1}\text{area}(Q)\\
&\leq \frac{2C^{(5)}_{U,O} \left(1+5^{|1-\gamma_O|_0}\right)}{\sin^2(\phi)} \sqrt{\sum\limits_{Q \in T_{O,h}} |\nabla_Q f|^2 \text{area}(Q)}\sqrt{\sum\limits_{Q \in T_{O,h}}|Oz_Q|^{2\gamma_O-2}\text{area}(Q)}\\
&\leq \frac{2C^{(5)}_{U,O} \left(1+5^{|1-\gamma_O|_0}\right)}{\sin^2(\phi)} \|f\|_{T_{O,h}}\sqrt{\sum\limits_{Q \in T_{O,h}}|Oz_Q|^{2\gamma_O-2}\text{area}(Q)}.
\end{align*}

Introducing polar coordinates $(r,\psi)$ around $O$ again, we can use for the latter term the results that we obtained in the proofs of Lemmas~\ref{lem:gradient_inside} and~\ref{lem:contribution_inside}. Thus,
\[\sum\limits_{Q \in T_{O,h}}|Oz_Q|^{2\gamma_O-2}\text{area}(Q)\leq 2^{4+2|\gamma_O-1|_0}\sum\limits_{Q \in T_{O,h}}\int_Q r^{2\gamma_O-2}rdrd\psi\leq 2^{4+2|\gamma_O-1|_0}\frac{3^{2\gamma_O}\pi}{\gamma_O^2}h^{2\gamma_O}.\]

The results now follow from the inequality of the arithmetic and the quadratic mean for the $\|f\|_{T_{O,h}}$ and $h \leq C^{(8)}$ in a similar way as in the proof of Proposition~\ref{prop:approximation_error}. The $h$-adapted case follows similarly.
\end{proof}

By applying Corollary~\ref{cor:contribution_outside3}, Lemma~\ref{lem:contribution_inside2}, and Lemma~\ref{lem:contribution_inside2}, we can bound the three sums in Equation (\ref{eq:consistency}), leading to an estimate for the consistency error in Lemma~\ref{lem:strang}.

\begin{proposition}\label{prop:consistency_error}
For a given function $f:V(\Lambda)\to\mathds{R}$, if the maximum edge length $h$ of the $\phi$-regular discretization $(\Sigma,\Lambda)$ satisfies $h \leq C^{(8)}$ or $h \leq C^{(8,h)}$ in the $h$-adapted case, then the consistency error can be bounded as follows:
\begin{align*}
a(u,f)&\leq C^{(12)}_{U,\phi}\sqrt{\lambda_\Sigma(h)}\left\|f\right\|\\
\textnormal{and} \quad a(u,f)&\leq C^{(12,h)}_{U,\phi}\sqrt{h}\left\|f\right\| \quad \textnormal{in the $h$-adapted case}.
\end{align*}
\end{proposition}


\subsection{The convergence theorems}\label{sec:convergence_theorems}

We now present the convergence theorem that summarizes the results from the previous two sections. This theorem is analogous to Lemma 4.16 in \cite{BoSk12}, adapted to our context.

\begin{theorem}\label{th:energy_convergence}
Let $P \in \mathds{R}^{2g}$ be a given vector of periods. Consider $\Omega$, the unique smooth holomorphic differential with real parts of its $a$- and $b$-periods given by $P$, and $\omega$, the unique discrete holomorphic differential with real parts of its black and white $a$- and $b$-periods given by $P$, such that corresponding black and white $a$- and $b$-periods coincide. Let $U:=\int \re(\Omega)$ and $u:=\int \re(\omega)$ be the corresponding multi-valued (discrete) harmonic functions.

For any $\phi$-regular discretization $(\Sigma,\Lambda)$ with a maximal side length $h$, the following holds:
\begin{align*}
\frac{1}{2}\left|\langle \Omega,\Omega \rangle -\langle \omega,\omega \rangle \right|=\left| \|U\|^2-\|u\|^2 \right\| &\leq C^{(13)}_{P,\phi} \lambda_\Sigma(h)\quad \textnormal{if $h \leq C^{(8)}$}\\
\textnormal{and} \quad \frac{1}{2}\left|\langle \Omega,\Omega \rangle -\langle \omega,\omega \rangle \right|=\left| \|U\|^2-\|u\|^2 \right\| &\leq C^{(13,h)}_{P,\phi}h\quad \textnormal{if $(\Sigma,\Lambda)$ is $h$-adapted and $h \leq C^{(8,h)}$}.
\end{align*}
\end{theorem}

\begin{proof}
By applying the Second Lemma of Strang, Lemma~\ref{lem:strang}, along with the estimates on the approximation and consistency errors in Propositions~\ref{prop:approximation_error} and~\ref{prop:consistency_error}, we deduce the inequality
\[\|U-u\| \leq 2 \inf\limits_{v \in M_P} \left\|U-v\right\|+\sup\limits_{f \in M_0} \frac{\left|a(U,f)\right|}{\left\|f\right\|} \leq  2 \sqrt{C^{(9)}_{U,\phi}\lambda_\Sigma(h)}+ C^{(12)}_{U,\phi}\sqrt{\lambda_\Sigma(h)}.\]

Now, using the Cauchy-Schwarz inequality for the positive definite scalar product $a$, we have
\begin{align*}
\left|\|U\|^2- \|u\|^2\right|&=\left|a(U-u,U+u)\right|=\left|a(U-u,U-u)+2a(U-u,U)\right|\\
&\leq \|U-u\|^2+2|a(U-u,U)|\leq \|U-u\|^2+2\|U-u\|\|U\|\\
&\leq \left(2 \sqrt{C^{(9)}_{U,\phi}}+ C^{(12)}_{U,\phi}\right)^2\lambda_\Sigma(h) + 2\left(2 \sqrt{C^{(9)}_{U,\phi}}+ C^{(12)}_{U,\phi}\right)\|U\|\sqrt{\lambda_\Sigma(h)}\\
&\leq \left(4 \left(C^{(9)}_{U,\phi}+\sqrt{C^{(9)}_{U,\phi}}\|U\|\right) + (C^{(12)}_{U,\phi})^2+ 2\|U\| C^{(12)}_{U,\phi}\right) \lambda_\Sigma(h)
\end{align*}
since $h \leq C^{(8)} \leq 1$. The argument for $h$-adapted discrete Riemann surfaces follows similarly.
\end{proof}

The following corollary, which will be useful in Section~\ref{sec:convergence2} when we establish the boundedness of the discrete energy form $E_\Lambda$, is presented:

\begin{corollary}\label{cor:energy_convergence2}
Let $P^B,P^W \in \mathds{R}^{2g}$ be two given vectors of periods. Let $\omega'$ be the unique discrete holomorphic differential with the real parts of its black and white $a$- and $b$-periods given by $P^B$ and $P^W$, respectively. Let $u':=\int \re(\omega')$ be the corresponding multi-valued discrete harmonic function. Then, $\frac{1}{2}\langle \omega', \omega' \rangle= |u'|^2 \leq C^{(14)}_{P^B,P^W,\phi}$ holds true for any $\phi$-regular discretization $(\Sigma,\Lambda)$ that satisfies $h \leq C^{(8)}$.
\end{corollary}

\begin{proof}
We define multi-valued harmonic functions $U^B$ and $U^W$ with periods $P^B$ and $P^W$, respectively. Restricting them to $V(\tilde{\Lambda})$, we denote these restrictions as $U^B_{\Lambda}$ and $U^W_{\Lambda}$. We then introduce the discrete function $U^{B,W}{\Lambda}:V(\tilde{\Lambda}) \to \mathds{R}$, which agrees with $U^B{\Lambda}$ on black vertices and with $U^W_{\Lambda}$ on white vertices.

As $u'$ is a discrete harmonic function, and $U^{B,W}_{\Lambda}$ shares the same periods, we have $|u'|^2\leq|U^{B,W}_{\Lambda}|^2$ by Corollary~\ref{cor:harmonic}. In a chart $z_Q$ of a given quadrilateral $Q \in F(\Lambda)$, let's denote $b$ and $w \in \mathds{C}$ as unit vectors parallel to the black and white diagonal, respectively. Due to the construction, we find that $\langle \nabla_Q U^{B,W}, b \rangle = \langle \nabla_Q U^{B}, b \rangle$ and $\langle \nabla_Q U^{B,W}, w \rangle = \langle \nabla_Q U^{W}, w \rangle$. Therefore, we can bound $\left|\nabla_Q U^{B,W}_{\Lambda}\right|^2$ as follows:
\[\left|\nabla_Q U^{B,W}_{\Lambda}\right|^2\leq \frac{4}{\sin^2{\phi}}\max\left\{(\nabla_Q U^{B} \cdot b )^2, ( \nabla_Q U^{W} \cdot w )^2\right\} \leq \frac{4}{\sin^2{\phi}}\left( \left|\nabla_Q U^{B}\right|^2+ \left| \nabla_Q U^{W} \right|^2\right)\]
due to Lemma~\ref{lem:projection}. By integrating over $\Sigma$ and using Theorem~\ref{th:energy_convergence}, we obtain the inequality:
\begin{align*}
\|u'\|^2&\leq\|U^{B,W}_{\Lambda}\|^2\leq \frac{4}{\sin^2{\phi}}\left(\|U^{B}_{\Lambda}\|^2+\|U^{W}_{\Lambda}\|^2\right)\leq \frac{4}{\sin^2{\phi}}\left(\|U^{B}\|^2+\|U^{B}-U^{B}_{\Lambda}\|^2+\|U^{W}\|^2+\|U^{W}-U^{W}_{\Lambda}\|^2\right)\\
&\leq \frac{4}{\sin^2{\phi}}\left(\|U^{B}\|^2+\|U^{W}\|^2+ C^{(13)}_{P_B,\phi} +C^{(13)}_{P_W,\phi}\right)=:C^{(14)}_{P^B,P^W,\phi}. \qedhere
\end{align*}
\end{proof}

Now we introduce some necessary notions to define the convergence of the discrete Dirichlet energy to its smooth counterpart as quadratic forms.

\begin{definition}
Let $M=\left(M_{ij}\right)_{i,j=1}^k$ be a real $k\times k$-matrix. We denote the Frobenius norm of $M$ by \[\left\|M\right\|:= \sqrt{\sum\limits_{i,j=1}^k M_{ij}^2}.\]
\end{definition}

\begin{remark}
The Frobenius norm is sub-multiplicative, $\left\|MN\right\|\leq \left\|M\right\| \cdot \left\|N\right\|$, for any two real $k \times k$-matrices $M,N$. If $M$ is a real number, then $\left\|M\right\|=|M|$.
\end{remark}

\begin{definition}
For a discretization $(\Sigma,\Lambda)$ of the Riemann surface $\Sigma$, we denote by $h$ the maximal edge length of the embedded quad-graph $\Lambda$. Let $X_\Lambda$ be a square matrix depending on the discretization $(\Sigma,\Lambda)$ (including dependency on $\Sigma$), and $X_\Sigma$ a matrix of the same size depending solely on the geometry of $\Sigma$. We write $X_\Lambda \to X_\Sigma$ as $h\to 0$ if there exist two constants $c$ and $C$, depending only on $\Sigma$ and the angle parameter $\phi$, such that $\left\| X_\Lambda - X_\Sigma \right\| \leq C\lambda_\Sigma(h)$ holds true for any $\phi$-regular discretization $(\Sigma,\Lambda)$ satisfying $h \leq c$. In the case of $h$-adapted discretizations, we replace $\lambda_\Sigma(h)$ with $h$.
\end{definition}

The convergence of energies in Theorem~\ref{th:energy_convergence} implies the convergence of the quadratic forms.

\begin{theorem}\label{th:quadratic_form_convergence}
If the maximal edge length of a sequence of $\phi$-regular discrete Riemann surfaces on $\Sigma$ approaches zero, then the quadratic form corresponding to the Dirichlet energy of discrete harmonic forms with equal black and white periods converges to its smooth counterpart, $E_\Lambda^{B=W} \to E_\Sigma$ as $h\to 0$.
\end{theorem}

\begin{proof}
Theorem~\ref{th:energy_convergence} ensures that for any given vector $P \in \mathds{R}^{2g}$, \[\left|P^T E_\Lambda^{B=W}P-P^T E_\Sigma P\right|=2\left|\langle U,U \rangle-\langle u,u \rangle \right|\leq C^{(13)}_{P,\phi} \lambda_\Sigma(h)\] for all discretizations $(\Sigma,\Lambda)$ with maximum edge length $h\leq C^{(8)}$. Here, $U$ and $u$ are the corresponding multi-valued smooth and discrete harmonic functions with periods given by $P$. The dependency of $C^{(13)}_{P,\phi}$ on $P$ corresponds to the dependency on $\|U\|$ and $\max_{z \in \overline{g_O(D_O)}}\|D^kU_O(z)\|$ for $k=1,2$ and charts $(D_O,g_O)$, $O \in S'$, used in the proof of Proposition~\ref{prop:approximation_error}. As $U$ depends smoothly on $P$, $C^{(13)}_{P,\phi}$ depends continuously on $P$, and we denote $C^{(15)}_{\phi}:=\max_{|P|=1} C_P$. It follows that \[\left|P^T (E_\Lambda^{B=W}- E_\Sigma) P\right|\leq C^{(15)}_{\phi} \lambda_\Sigma(h)\|P\|^2\] for all $P \in \mathds{R}^{2g}$ provided that $h \leq C^{(8)}$. Finally, we show that $|E_\Lambda^{B=W}- E_\Sigma|\leq \sqrt{2g} C^{(15)}_{\phi}\lambda\Sigma(h)$ by observing that the difference is diagonalizable, and the absolute value of any eigenvalue is bounded by $C^{(15)}_{\phi} \lambda\Sigma(h)$. The $h$-adapted case is treated similarly.
\end{proof}

In conclusion, we examine the convergence of multi-valued discrete holomorphic functions in the context of nondegenerate uniform sequences of orthodiagonal discrete Riemann surfaces $(\Sigma,\Lambda_n)$. This convergence mirrors the result presented in Theorem 5.3 by Bobenko and Skopenkov \cite{BoSk12}. Their proof relies on a similar statement established by Skopenkov in \cite{Sk13} for planar orthodiagonal quadrilateral lattices. The key transformation of a Delaunay triangulation into a Delaunay-Voronoi quadrangulation, discussed in \cite{BoSk12,Sk13}, plays a pivotal role in the argument. Furthermore, the adaptations made to handle conical singularities in \cite{BoSk12} naturally extend to the quadrilateral case.

In light of these connections and to avoid redundancy, we opt not to reproduce the proof of the theorem here. For interested readers seeking a detailed proof, we refer them to the upcoming master's thesis of Maximilian Pschigode (Technische Universit\"at Berlin).

\begin{definition}
We introduce the notion of a \textit{nondegenerate uniform sequence} of discrete Riemann surfaces $(\Sigma, \Lambda_n)$, assuming an atlas ${(D_O, g_O)}_{O \in S}$ of $\Sigma$ and an angle parameter $\phi>0$. To be classified as \textit{nondegenerate uniform}, the sequence must satisfy two conditions:
\begin{itemize}
\item The discrete Riemann surfaces $(\Sigma, \Lambda_n)$ must be $\phi$-regular.
\item For any point $p$ that does not lie within any $D_O$ associated with a singularity $O \in S$, the number of vertices in $V(\Lambda_n)$ within any disk centered at $p$ and having a radius equal to the maximal edge length must be less than a fixed constant $C$ that is independent of $n$. If $p$ does lie within a $D_O$ for some $O \in S$, the same condition should apply, or alternatively, the number of vertices $g_O(V(\Lambda_n \cap D_O))$ within any disk centered at $g_O(p)$ and having a radius equal to the maximal edge length should be smaller than $C$.
\end{itemize}
\end{definition}

It is worth noting that the boundedness of the interior angles of the quadrilaterals implies a bounded ratio of the lengths of their diagonals, which is the condition effectively used in \cite{Sk13}. The modification of the second condition for points close to a conical singularity arises from the $h$-adapted case and was introduced by Bobenko and B\"ucking in \cite{BoBu17}.

\begin{theorem}\label{th:harmonic_convergence}
Let $(\Sigma,\Lambda_n)$ be a sequence of nondegenerate uniform orthodiagonal discrete Riemann surfaces, such that the maximal edge length converges to zero as $n$ tends to infinity. Consider $P_n \in \mathds{R}^{2g}$, a sequence of vectors converging to a vector $P$. Let $(p_n,p'_n) \in E(\Lambda_n)$ be a sequence of edges such that $p_n \to p \in \Sigma$ as $n \to \infty$.

Let $U:\tilde{\Sigma} \to \mathds{R}$ be the unique multi-valued harmonic function with periods given by $P$, satisfying $U(p)=0$. Let $u_n:V(\tilde{\Lambda}_n)\to \mathds{R}$ be the unique real multi-valued discrete harmonic function with equal black and white periods given by $P_n$, satisfying $u_n(p_n)=u_n(p'_n)=0$.

Then, the functions $u_n$ converge to $U$ uniformly on each compact subset of $\tilde{\Sigma}$ as $n$ approaches infinity.
\end{theorem}


\section{Convergence of the discrete period matrix}\label{sec:convergence2}

Before delving into the convergence of discrete period matrices to their smooth counterpart in Theorem~\ref{th:matrix_convergence}, we shall first introduce two lemmas that will be essential for the proof.

\begin{lemma}\label{lem:definiteness}
Let $M:= \left(\begin{matrix} A & B\\ B^T & C\end{matrix}\right)$ be a symmetric and positive definite $2g \times 2g$-matrix with the $g \times g$-blocks $A=A^T, B, C=C^T$. Consider the $g \times 2g$-matrix $L:=\left( 1_g \quad 1_g \right)$ consisting of twice the identity matrix.

Then, the matrix $L M L^T-4 \left(L M^{-1} L^T\right)^{-1}$ is positive semidefinite. Moreover, if this matrix approaches zero and $M$ remains bounded throughout the convergence, then $A-C$ and $B-B^T$ also approach zero.
\end{lemma}

We would like to express our gratitude to G\"unter Rote for providing us with the proof of positive semidefiniteness.

\begin{proof}
Let $J:= \frac{1}{\sqrt{2}}\left(\begin{matrix} 1_g & 1_g\\ 1_g & -1_g\end{matrix}\right)$, which is an orthogonal and symmetric matrix ($J^{-1}=J^T=J$). We define $K:=\left( 1_g \quad 0_g \right)$, a $g \times 2g$-matrix with the identity matrix and the zero matrix.

By performing some algebraic manipulations, we obtain:
\[L M L^T-4 \left(L M^{-1} L^T\right)^{-1}=2 KJ M JK^T-4 \left(2 KJ M^{-1} JK^T\right)^{-1}=2K\hat{M}K^T-2 \left(K \hat{M}^{-1} K^T\right)^{-1},\]
where $\hat{M}:=J M J$. Since $M$ is positive definite, $\hat{M}:= \left(\begin{matrix} \hat{A} & \hat{B}\\ \hat{B}^T & \hat{C}\end{matrix}\right)$ is also positive definite. Specifically, the upper left block of $\hat{M}^{-1}$ is given by the inverse of the Schur complement $\hat{M}/\hat{C}=\hat{A}-\hat{B}^T \hat{C}^{-1}\hat{B}$, and thus $K \hat{M}^{-1} K^T=\left(\hat{A}-\hat{B}^T \hat{C}^{-1}\hat{B}\right)^{-1}.$ This simplifies the expression to:
\[L M L^T-4 \left(L M^{-1} L^T\right)^{-1}=2\hat{A}-2 \left(\hat{A}-\hat{B}^T \hat{C}^{-1}\hat{B}\right)=2\hat{B}^T \hat{C}^{-1}\hat{B}.\]

Since $\hat{M}$ and its positive definite diagonal block $\hat{C}$ remain bounded during convergence, we know that $\hat{C}^{-1}$ is also bounded away from zero. Thus, $\hat{B}^T \hat{C}^{-1}\hat{B}$ is positive semidefinite.

By setting $Q$ as an orthogonal matrix such that $Q\hat{C}^{-1}Q^T=:D$ is diagonal, and defining $R:= Q^T \sqrt{D} Q$, we have $R^2=\hat{C}^{-1}$. Consequently, $\hat{B}^T \hat{C}^{-1}\hat{B}=\left(R\hat{B}\right)^TR\hat{B}$. The latter matrix represents the scalar products of the entries of $R\hat{B}$. Therefore, if $\hat{B}^T \hat{C}^{-1}\hat{B}$ approaches zero, then $R\hat{B}$ does as well. As no eigenvalue of the positive definite matrix $R$ goes to zero, we conclude that $\hat{B}$ approaches zero.

As a result of this construction, we have $2\hat{B}=A-C+B^T-B$. Since $A-C$ is symmetric and $B^T-B$ is antisymmetric, if their sum approaches zero, then each summand does as well.
\end{proof}

\begin{lemma}\label{lem:boundedness_periodmatrix}
Suppose the maximal edge length of $(\Sigma,\Lambda)$ satisfies $h\leq C^{(8)}$. Then, $\|E_\Lambda\|\leq C^{(16)}_{\phi}$.
\end{lemma}

\begin{proof}
Let us decompose $\mathds{R}^{4g}=\mathds{R}^{2g}\oplus\mathds{R}^{2g}$ into a black and a white part, and denote the corresponding decompositions of vectors as $P'=P^B \oplus P^W$. By Corollary~\ref{cor:energy_convergence2}, for any such vector $P'$, we know that $\left|P'^T E_\Lambda P' \right|\leq C^{(14)}_{P^B,P^W,\phi}$. Since $E_\Lambda$ is symmetric, we can find an orthonormal basis of $\mathds{R}^{4g}$ consisting of eigenvectors $e_i$. Let $C$ be the maximum of all $C^{(14)}_{e_i,\phi}$. Then, for all $P'$ with $|P'|\leq 1$, it follows that $\left|P'^T E_\Lambda P' \right|\leq C$. Consequently, $|E_\Lambda|\leq \sqrt{2g} C=:C^{(16)}_{\phi}$.
\end{proof}

\begin{theorem}\label{th:matrix_convergence}
As the maximal edge lengths of a sequence of $\phi$-regular discrete Riemann surfaces on $\Sigma$ tend to zero, the diagonal blocks of the complete discrete period matrix converge to the same matrix, the off-diagonal blocks converge to the same symmetric matrix, and the sum of a diagonal and an off-diagonal block converges to the period matrix of $\Sigma$:
\[\Pi^{B,W}- \Pi^{W,B}, \Pi^{B,B}- \Pi^{W,W} \to 0 \quad \text{and} \quad \Pi^{B,W}+\Pi^{B,B}, \Pi^{W,B}+\Pi^{W,W} \to \Pi_\Sigma  \qquad \text{as } h\to 0.\]

In particular, the discrete period matrices converge to their smooth counterpart: $\Pi \to \Pi_\Sigma$ as $h \to 0$.
\end{theorem}

\begin{proof}
As $h \to 0$, the matrix $E_\Lambda^{B=W}$ converges to $E_\Sigma$ by Theorem~\ref{th:energy_convergence}. Let $L:=\left( 1_g \quad 1_g \right)$, and using the representations of $E_\Sigma$ and $E_\Lambda$ provided in Lemmas~\ref{lem:energy_continuous} and~\ref{lem:energy_discrete}, as well as the fact that convergence of a matrix to zero in the Frobenius norm implies the convergence of any submatrix to zero, we obtain:
\begin{align}
L \left( \im \tilde{\Pi} \right)^{-1} L^T &\to 2\left( \im \Pi_{\Sigma} \right)^{-1}, \label{eq:1}\\
L  \left( \im \tilde{\Pi} \right)^{-1} \re \tilde{\Pi} L^T &\to 2\left( \im \Pi_{\Sigma} \right)^{-1} \re \Pi_{\Sigma}, \label{eq:2}\\
L \re \tilde{\Pi}  \left( \im \tilde{\Pi} \right)^{-1}  L^T &\to 2\re \Pi_{\Sigma} \left( \im \Pi_{\Sigma} \right)^{-1}, \label{eq:3}\\
L \re \tilde{\Pi} \left( \im \tilde{\Pi} \right)^{-1} \re \tilde{\Pi} L^T + L \im \tilde{\Pi} L^T &\to 2\re \Pi_{\Sigma} \left( \im \Pi_{\Sigma} \right)^{-1} \re \Pi_{\Sigma} + 2\im \Pi_{\Sigma}. \label{eq:4}
\end{align}
From the Limits~(\ref{eq:1}), (\ref{eq:2}), and (\ref{eq:3}), we deduce the following:
\begin{align}
\re \Pi_{\Sigma} L \left( \im \tilde{\Pi} \right)^{-1}  L^T \re \Pi_{\Sigma} &\to 2\re \Pi_{\Sigma} \left( \im \Pi_{\Sigma} \right)^{-1} \re \Pi_{\Sigma}, \label{eq:5}\\
\re \Pi_{\Sigma} L  \left( \im \tilde{\Pi} \right)^{-1}  \re \tilde{\Pi} L^T &\to 2\re \Pi_{\Sigma} \left( \im \Pi_{\Sigma} \right)^{-1} \re \Pi_{\Sigma}, \label{eq:6}\\
L \re \tilde{\Pi}  \left( \im \tilde{\Pi} \right)^{-1} L^T \re \Pi_{\Sigma} &\to 2\re \Pi_{\Sigma} \left( \im \Pi_{\Sigma} \right)^{-1} \re \Pi_{\Sigma}. \label{eq:7}
\end{align}
Taking the sum of the Limits~(\ref{eq:4}) and~(\ref{eq:5}) and subtracting~(\ref{eq:6}) and~(\ref{eq:7}), we get:
\begin{align}
\left(L \re \tilde{\Pi}-\re \Pi_{\Sigma} L\right) \left( \im \tilde{\Pi} \right)^{-1} \left(L \re \tilde{\Pi}-\re \Pi_{\Sigma} L\right)^T +L \im \tilde{\Pi} L^T &\to 2\im \Pi_{\Sigma}. \label{eq:8}
\end{align}
The latter can be restated as:
\begin{align}
\left(L \re \tilde{\Pi}-\re \Pi_{\Sigma} L\right) \left( \im \tilde{\Pi} \right)^{-1} \left(L \re \tilde{\Pi}-\re \Pi_{\Sigma} L\right)^T +L \im \tilde{\Pi} L^T - 4\left(L \left(\im \tilde{\Pi}\right)^{-1} L^T \right)^{-1} &\to 0 \label{eq:9}
\end{align}
provided that we show
\begin{align}\label{eq:convergence_lemma}
2\left(L \left(\im \tilde{\Pi}\right)^{-1} L^T \right)^{-1}\to \im \Pi_\Sigma.
\end{align}

To establish \eqref{eq:convergence_lemma}, we observe from the Limit~(\ref{eq:1}) that $2\delta:=L \left(\im \tilde{\Pi}\right)^{-1} L^T - 2 \left( \im \Pi_{\Sigma} \right)^{-1} \to 0.$ We set $T:= 1_g - \im \Pi_\Sigma \left( \left( \im \Pi_{\Sigma} \right)^{-1} + \delta \right)$, which satisfies $|T|=|\im \Pi_{\Sigma} \delta| \leq |\im \Pi_{\Sigma}| \cdot |\delta|<\frac{\varepsilon}{2 |\im \Pi_{\Sigma}|}<\frac{1}{2}$, provided that $0<\varepsilon< |\im \Pi_{\Sigma}|$ and $h$ is small enough to ensure $2|\delta| <\varepsilon |\im \Pi_{\Sigma}|^{-2}$. In particular, the Neumann series $\sum_{k=0}^{\infty} T^k$ converges and yields the inverse of $1_g-T=\im \Pi_\Sigma \left( \left( \im \Pi_{\Sigma} \right)^{-1} + \delta \right)$. Thus,
\begin{align*}
\left\|2\left(L \left(\im \tilde{\Pi}\right)^{-1} L^T\right)^{-1} -\im \Pi_\Sigma\right\|&=\left\|\left( \left( \im \Pi_{\Sigma} \right)^{-1} + \delta \right)^{-1}-\im \Pi_\Sigma\right\|\\
&=\left\|\left(1_g-T\right)^{-1}\im \Pi_\Sigma  -\im \Pi_\Sigma\right\|=\left\|\sum\limits_{k=1}^{\infty} T^k\im \Pi_\Sigma \right\|\\
&=\left\|T \left(1_g-T\right)^{-1}\im \Pi_\Sigma\right\|\leq \frac{\left\|T\right\| \cdot \left\|\im \Pi_\Sigma \right\|}{1-\left\|T\right\|}<\varepsilon
\end{align*}
for any small enough $\varepsilon$. Hence, \eqref{eq:convergence_lemma} indeed holds true.

Since $\im \tilde{\Pi}$ is positive definite by Theorem~\ref{th:period_matrix}, $\left(L \re \tilde{\Pi}-\re \Pi_{\Sigma} L\right) \left( \im \tilde{\Pi} \right)^{-1} \left(L \re \tilde{\Pi}-\re \Pi_{\Sigma} L\right)^T$ is positive semidefinite. Also, $L \im \tilde{\Pi} L^T - 4\left(L \left(\im \tilde{\Pi}\right)^{-1} L^T \right)^{-1}$ is positive semidefinite by Lemma~\ref{lem:definiteness}. We have shown in \eqref{eq:9} that the sum of these two positive semidefinite matrices approaches zero. By multiplying with corresponding eigenvectors, we deduce that any eigenvalue of each of the two matrices converges to zero. Since the Frobenius norm is invariant under multiplication with orthogonal matrices, it follows that each individual matrix converges to zero.

Using the boundedness of $E_\Lambda$ by Lemma~\ref{lem:boundedness_periodmatrix}, we can apply Lemma~\ref{lem:definiteness} to deduce that \[\im \Pi^{B,W}-\im \Pi^{W,B} \to 0 \quad \text{and} \quad \im \Pi^{B,B}-\im \Pi^{W,W} \to 0.\] From (\ref{eq:convergence_lemma}), we further deduce that
\begin{align}
L \im \tilde{\Pi} L^T \to 4\left(L \left(\im \tilde{\Pi}\right)^{-1} L^T \right)^{-1} &\to 2\im \Pi_\Sigma, \label{eq:11}\\ \textnormal {so } \quad \im \Pi^{B,W}+ \im \Pi^{B,B} \quad \textnormal{and} \quad \im \Pi^{W,B}+ \im \Pi^{W,W} &\to \im \Pi_\Sigma. \notag
\end{align}

In analogy to the proof of Lemma~\ref{lem:definiteness}, Limits~\eqref{eq:7} and \eqref{eq:11}, together with the boundedness of the positive definite matrix $\im \tilde{\Pi}$, imply that $\left(L \re \tilde{\Pi}-\re \Pi_{\Sigma} L\right) \to 0.$ In particular,  \[\re \Pi^{B,W}+ \re \Pi^{W,W} \to \re \Pi_\Sigma \quad \text{and} \quad \re \Pi^{W,B}+ \re \Pi^{B,B} \to \re \Pi_\Sigma.\]

Taking the difference, we have $(\re \Pi^{B,W}-\re \Pi^{W,B})+ (\re \Pi^{W,W}- \re \Pi^{B,B}) \to 0$. This is the sum of a symmetric and an antisymmetric matrix, so each summand goes to zero. Thus, \[\re \Pi^{B,W}-\re \Pi^{W,B} \to 0 \quad \text{and} \quad \re \Pi^{W,W}- \re \Pi^{B,B} \to 0. \qedhere\]
\end{proof}

In the case of discrete Riemann surfaces based on quad-graphs with orthogonal diagonals, Theorem~\ref{th:matrix_convergence} simplifies to the following result:

\begin{corollary}\label{cor:matrix_convergence_orthogonal}
For a sequence of $\phi$-regular orthodiagonal discrete Riemann surfaces on $\Sigma$ with maximal edge lengths approaching zero, the blocks of the complete discrete period matrix converge to the real and imaginary parts of the period matrix of $\Sigma$:
\[\Pi^{B,W}, \Pi^{W,B} \to \im \Pi_\Sigma \qquad \text{and } \qquad \Pi^{B,B}, \Pi^{W,W} \to \re \Pi_\Sigma  \qquad \text{as } h\to 0.\]
\end{corollary}

\begin{proof}
By Lemma~\ref{lem:period_orthodiagonal}, $\Pi^{B,W}$ and $\Pi^{W,B}$ are purely imaginary, and $\Pi^{B,B}$ and $\Pi^{W,W}$ are purely real. The convergence result from Theorem~\ref{th:matrix_convergence} implies that $\Pi^{B,W},\Pi^{W,B}$ converge to the imaginary part of $\Pi_\Sigma$, and $\Pi^{B,B},\Pi^{W,W}$ converge to its real part.
\end{proof}

We aim to deduce the convergence of discrete Abelian integrals of the first kind from Corollary~\ref{cor:matrix_convergence_orthogonal} and Theorem~\ref{th:harmonic_convergence}, but directly applying Theorem~\ref{th:harmonic_convergence} to the real and imaginary parts of the multi-valued discrete holomorphic function is not possible since, in general, the black and white $b$-periods of a discrete holomorphic differential with equal black and white $a$-periods are not equal.

However, Corollary~\ref{cor:matrix_convergence_orthogonal} allows us to deduce that these black and white $b$-periods become equal in the limit. By decomposing $P' \in \mathds{R}^{4g}$ into its black and white parts as $P'=P^B \oplus P^W$, we can substitute the matrix $\left(\begin{matrix} \im \Pi_\Sigma & \re \Pi_\Sigma\\ \re \Pi_\Sigma & \im \Pi_\Sigma\end{matrix}\right)$, which represents the complete discrete period matrix obtained in the limit, into the quadratic form $E_\Lambda$ from Lemma~\ref{lem:energy_discrete}. With the representation of $E_\Sigma$ given in Lemma~\ref{lem:energy_discrete}, we readily verify that $2P'^T E_\Lambda P$ converges to $(P^B)^T E_\Sigma P^B + (P^W)^T E_\Sigma P^W$ as $h \to 0$.

These two observations imply that in the orthodiagonal case, the discrete Dirichlet energy of the discrete harmonic differential with black periods given by $P^B$ and white periods given by $P^W$ approaches the Dirichlet energy of the harmonic differential with periods given by $P$, provided the maximum edge length of the $\phi$-regular discrete Riemann surface goes to zero and $P^B$ and $P^W$ both converge to $P$. This property is crucial in the original proof of Theorem~\ref{th:harmonic_convergence} by Bobenko and Skopenkov \cite{BoSk12}. Therefore, employing Corollary~\ref{cor:matrix_convergence_orthogonal}, we can also apply Theorem~\ref{th:harmonic_convergence} to the sequences of real and imaginary parts of discrete Abelian integrals of the first kind with equal black and white $a$-periods. This leads us to the following convergence theorem.

\begin{theorem}\label{th:holomorphic_convergence}
Consider a nondegenerate uniform sequence of orthodiagonal discrete Riemann surfaces $(\Sigma,\Lambda_n)$, where the maximal edge length converges to zero as $n \to \infty$. Let $\mathcal{A} \in \mathds{C}^{g}$ be a given vector of $a$-periods, and let $(p_n,p'_n) \in E(\Lambda_n)$ be a sequence of edges such that $p_n \to p \in \Sigma$ as $n \to \infty$. Denote by $\Omega$ the unique holomorphic differential with $a$-periods given by $\mathcal{A}$, and denote by $\omega_n$ the unique holomorphic differential on the medial graph of $\Lambda_n$ with equal black and white $a$-periods given by $\mathcal{A}$. We normalize the Abelian integral of the first kind $\int \Omega$ to be zero at $p$, and $\int \omega_n$ is normalized such that $\int \omega_n(p_n)=\int \omega_n(p'_n)=0$. Then, the discrete Abelian integrals of the first kind $\int \omega_n$ converge to $\int \Omega$ uniformly on each compact subset of $\tilde{\Sigma}$.
\end{theorem}


\section{Conclusion}\label{sec:conclusion}

The main result of our paper, Theorem~\ref{th:matrix_convergence}, establishes the convergence of the discrete period matrix to its continuous counterpart for sequences of rather general quadrangulations. We have also specified the limits of the four blocks of the complete discrete period matrix, shedding light on the relation between black and white periods. While in the case of orthodiagonal quadrilaterals, these blocks coincide with the real and imaginary parts of the continuous period matrix, this relation does not hold in the general case, as can be easily verified. Surprisingly, the limit matrix exhibits complete symmetry when black and white periods are interchanged, a notable result considering the general lack of such symmetry at the discrete level. The meaning and implications of this symmetry in the limit, particularly in statistical physics, remain intriguing open questions.

Our proof is conceptually distinct from Bobenko and Skopenkov's approach in \cite{BoSk12}, despite some similarities. In our proof, neither the black nor the white subgraph plays a specific role, and they can be interchanged. For numerical computations of discrete period matrices of compact polyhedral surfaces, Bobenko and Skopenkov's convergence result \cite{BoSk12} and its improved error rate using adapted triangulations \cite{BoBu17} are often sufficient, as triangulations are easy to obtain. Quadrangulations, however, offer natural discretizations for \textit{square-tiled surfaces} or \textit{origamis}, which are widely studied translation surfaces \cite{Z06, Sh22}. These quadrangulations, consisting of squares, allow for computationally easier parametrizations and subdivisions, enabling precise computation of discrete period matrices for surfaces with not too many vertices. This method was effectively applied by \c{C}elik, Fairchild, and Mandelshtam \cite{CFM23} to approximate the algebraic curve corresponding to a given translation surface. Our Theorem~\ref{th:matrix_convergence} not only provides theoretical confirmation of the convergence observed in \cite{CFM23} on square-lattices but also extends it to quadrangulations consisting of rectangles. This extension allows for a more straightforward approximation of the L-shape with an irrational ratio of side lengths. It also allows for considering affine transformations of square-tiled surfaces, leading to decompositions into parallelograms. The error estimates derived in our proof provide explicit expressions for constants, although in practice, their computation or approximation may be challenging due to their dependence on the derivatives of harmonic functions.

The discrete period matrices corresponding to the L-shaped translational surface obtained by \c{C}elik, Fairchild, and Mandelshtam are exclusively composed of purely imaginary elements. This outcome is expected since the L-shaped translational surface belongs to the class of M-curves, and all discretizations presented in \cite{CFM23} exhibit the same symmetry. Recently, D\"untsch's master's thesis \cite{D23} delved into this matter. She introduced the concept of discrete real Riemann surfaces and demonstrated that the discrete period matrix of a discrete real Riemann surface shares the same structure as the period matrix of the corresponding real Riemann surface with matching numbers of real ovals and type (dividing or non-dividing). Additionally, she provided characterizations of the complete discrete period matrices of discrete real Riemann surfaces.

However, characterizing the discrete Schottky locus, the space of discrete period matrices of compact discrete Riemann surfaces in the Siegel upper half space, and relating it to the classical Schottky locus remains an open question. Not all discrete Riemann surfaces correspond to polyhedral surfaces \cite{BoG17}, prompting us to explore whether the corresponding Schottky loci differ. Further research in this direction may offer new insights into discrete Riemann surfaces and their connections to continuous Riemann surfaces. 


\section*{Acknowledgment}

The author gratefully acknowledges the support from the Deutsche Forschungsgemeinschaft (DFG -- German Research Foundation) -- Project-ID 195170736 -- TRR109. Special thanks to Daniele Agostini, Paul Breiding, Samantha Fairchild, Yelena Mandelshtam, and T\"urk\"u \"Ozl\"um \c{C}elik for introducing the author to their research on the Schottky problem and demonstrating their interest in employing discrete methods for computing Riemann matrices of translational surfaces. The author also expresses gratitude to Alexander Bobenko and Ulrike B\"ucking for insightful discussions regarding previous approaches to address the convergence of discrete Riemann surfaces. Moreover, heartfelt thanks go to G\"unter Rote for resolving the issue of positive semidefiniteness in Lemma~\ref{lem:definiteness}.


\bibliographystyle{plain}
\begin{small}
\bibliography{Convergence}
\end{small}

\end{document}